\documentclass[a4paper,fleqn]{cas-sc}

\usepackage[authoryear,longnamesfirst]{natbib}
\usepackage{tikz}
\usepackage{hyperref}
\usepackage{mathtools}
\usetikzlibrary{shapes}
\usepackage{framed,xcolor}
\usepackage{setspace}
\usepackage{amsthm}

\def\tsc#1{\csdef{#1}{\textsc{\lowercase{#1}}\xspace}}
\tsc{WGM}
\tsc{QE}
\newtheorem{proposition}{Proposition}
\newdefinition{rmk}{Remark}
\newproof{pf}{Proof}
\newlength{\mywidth}
\begin{document}
\let\WriteBookmarks\relax
\def\floatpagepagefraction{1}
\def\textpagefraction{.001}

\shorttitle{An efficient mixed-integer linear programming formulation for solving influence diagrams}    

\shortauthors{Terho, Oliveira, Salo, Munari}  

\title [mode = title]{An efficient mixed-integer linear programming formulation for solving influence diagrams}  


 %Title footnote 1.
\tnotetext[1]{This work was supported by the Research Council of Finland (funding decision no. 332180) and the Profi-8 project funded by the Academy of Finland (funding decision no. 365215).}


\author[1]{Topias Terho}%[<options>]

%\cormark[1]


\ead{topias.terho@aalto.fi}


\credit{Investigation, Methodology, Software, Formal analysis, Writing – original draft, Writing – review \& editing}

\affiliation[1]{organization={Department of Mathematics and Systems Analysis, Aalto University},
    city={Espoo},
    % citysep={}, % Uncomment if no comma needed between city and postcode
    postcode={02150}, 
    country={Finland}}

\affiliation[2]{organization={DTU Management, Technical University of Denmark},
                addressline={}, 
    city={Kgs. Lyngby},
    % citysep={}, % Uncomment if no comma needed between city and postcode
    postcode={2800}, 
    country={Denmark}}

    \affiliation[3]{organization={Department of Production Engineering, Federal University of S\~ao Carlos},
    city={S\~ao Carlos - SP},
    % citysep={}, % Uncomment if no comma needed between city and postcode
    postcode={13565-905}, 
    country={Brazil}}

    \author[1,2]{Fabricio Oliveira}
\cormark[1]
\credit{Project administration, Supervision, Funding acquisition, Writing – original draft, Writing – review \& editing}

\ead{fabricio.oliveira@aalto.fi}

\author[1]{Ahti Salo}

\ead{ahti.salo@aalto.fi}

\credit{Supervision, Funding acquisition, Writing – original draft, Writing – review \& editing}

\author[3]{Pedro Munari}

\ead{munari@dep.ufscar.br}

\credit{Methodology, Writing – original draft, Writing – review \& editing}

%\ead[url]{https://www.aalto.fi/en/people/fabricio-oliveira}

\cortext[cor1]{Corresponding author}


\begin{abstract}
\protect\onehalfspacing
Influence diagrams represent decision-making problems with interdependencies between random events, decisions, and consequences. Traditionally, they have been solved using tailored algorithms that determine the expected utility-maximizing decision strategy. Towards this end, most recent state-of-the-art solution approaches convert them into mixed-integer linear programming (MILP) models that can be solved using off-the-shelf MILP solvers. Computationally, these approaches are most efficient for influence diagrams that represent periodic (or sequential) decision processes, which can be cast as partially observable Markov Decision Processes. However, they are inefficient for problems that lack this periodic structure or whose influence diagrams have large state spaces. 

The efficient MILP formulation presented in this paper is specifically designed for influence diagrams that are challenging for earlier MILP formulations. Additionally, this formulation can be adapted to optimize conditional value-at-risk and to incorporate chance and logical constraints, thereby retaining the modeling flexibility of MILP-based methods. Finally, computational experiments on a range of problems from the literature are reported to show that the proposed formulation is much more efficient than earlier MILP formulations, particularly in problems which cannot be modeled as partially observable Markov decision processes and which lead to influence diagrams whose decision or chance nodes may have large state spaces.  
\end{abstract}


%\item An efficient MILP formulation for solving influence diagrams is developed.
%\item The proposed formulation helps solve previously intractable influence diagrams.
%\item Its computational efficiency is assessed for different problem types. 
%\item Expected utility and conditional VaR can be included in the objective function.


% Keywords

\maketitle
\onehalfspacing 

% Main text
\section{Introduction}\label{sec:Intro}

%\begin{itemize}
%    \item Influence diagrams
%    \item Solution methods
%    \item Examples
%    \item Many influence diagrams represent problems where discretization matters
%    \item Need for formulations that can solve influence diagrams with fine discretization
%\end{itemize}

    Influence diagrams \citep{howard2005influenceretro,howard2005influence} are directed graphs that represent decision problems with dependencies between uncertain events, decisions, and consequences. They can be viewed as extensions of Bayesian networks that have been augmented by explicitly modeling the decision makers' decisions and the consequences that arise as a result of these decisions and the realization of chance events \citep{borgonovo2025fifty}. Depending on the context, influence diagrams have also been referred to as Bayesian decision networks \citep{gu2024bayesian, khosravi2020bayesian,lu2022bayesian, penman2020bayesian} or decision networks \citep{zhang1994computational}. 
    
    The popularity of influence diagrams in addressing decision problems can be attributed to their interpretability, sound methodological foundation, and wide availability of software tools. They provide an intuitive graphical representation of the problem, which fosters participation and facilitates communication between stakeholders \citep{borgonovo2025fifty}. Therefore,  influence diagrams are helpful in carrying out comprehensive modeling processes that start from the construction of a structured problem representation \citep{mcgovern1994using,regan1995r,lu2009interactive} and offer well-founded decision recommendations. For several decades, they have been widely employed in many application areas such as environmental resource management, software development, and modeling the societal impacts of human choice behavior \citep{oni2024evaluation,varis1990bayesian,virtanen2004modeling,wan2022stay,weflen2022influence}. 
    
    In general, solving the decision problem represented as an influence diagram gives the decision strategy that maximizes the expected utility of consequences for the decision maker. The early algorithms for finding this strategy exploited the principle of optimality in dynamic problems and, accordingly, were based on node removals and arc reversals \citep{shachter1986evaluating,smith1989influence} or the formulation of the equivalent decision tree representation, which was solved with dynamic programming \citep{tatman1990dynamic}. However, these algorithms entail restrictive assumptions, such as perfect recall, which means that all earlier decisions, as well as all information on the basis of which they were made, are remembered when making later ones.

    Motivated by the recognition that perfect recall as an assumption may not hold in real-life applications, \citet{lauritzen2001representing} proposed an approach to solve limited memory influence diagrams in which this assumption does not hold. Subsequently, several dedicated algorithms have been developed to determine the expected utility-maximizing strategy for limited memory influence diagrams. Examples include the single-policy-update algorithm \citep{lauritzen2001representing}, multiple-value-update \citep{maua2012solving}, and the \textit{k}-neighborhood search algorithm \citep{maua2016fast}. Notable modeling extensions include, for instance, influence diagrams with fuzzy random variables and fuzzy utilities \citep{rodriguez2005solving,lopez2007influence}, influence diagrams with multiple, possibly adversarial, decision makers \citep{koller2003multi,detwarasiti2005influence,gonzalez2019adversarial,roponen2020adversarial} and influence diagrams with asymmetric structures \citep{smith1993structuring,shenoy2000valuation,demirer2006sequential}.
    In what follows, the term influence diagram also covers those diagrams in which the assumption of perfect recall does not hold, as the formulation proposed in the present paper can be applied regardless of this assumption. 

    Decision Programming \citep{salo2022decision} and Rooted Junction Trees (RJT) \citep{parmentier2020integer} are two frameworks that transform influence diagrams into mixed-integer linear programming (MILP) models. They both exploit  the property that influence diagrams are essentially piecewise multilinear mappings from input parameters (probabilities and decisions) into corresponding outputs (expected utilities) \citep{borgonovo2014decision}. An advantage of the MILP frameworks over the early algorithms (e.g., node removals and arc reversals in influence diagrams with perfect recall) is that the resulting optimization models can be solved with efficient and ever-improving MILP solvers such as Gurobi \citep{gurobi}. 
    
    Another advantage of MILP-based frameworks is flexibility. For example, although risk aversion is generally modeled using a concave utility function in accordance with expected utility theory  \citep{borgonovo2018risk}, MILP formulations also admit alternative objective functions and technical constraints. Here, the conditional value-at-risk (CVaR) risk measure is particularly relevant, because it is coherent \citep{artzner1999coherent, rockafellar2000optimization} and a widely adopted standard in mathematical optimization \citep{ahmed2007coherent,liu2024newsvendor} and finance \citep{capinski2015hedging,nasini2022multi}. Specifically, CVaR represents the expected utility of consequences conditioned on the occurrence of the least preferred $\alpha$-tail of the consequence distribution, where $\alpha$ is the probability threshold for the tail. \citet{salo2022decision} present how CVaR can be optimized in the Decision Programming framework and \citet{herrala2025risk} show how this can be done in the RJT framework. In contrast, specialized algorithms that build on the principle of optimality are, to the best of our knowledge, generally suitable only for finding the expected utility-maximizing decision strategy.

    The RJT formulation for building a MILP model is based on a tree decomposition of the underlying influence diagram. Its computational efficiency depends on how the tree decomposition is derived, which in turn is conditioned by the structure of the influence diagram. In general, the RJT formulation performs remarkably well in decision problems with a periodic structure that can be cast as partially observable Markov Decision Processes (POMDPs). Examples of these problems include the chess problem in \citep{parmentier2020integer}, the pig farm problem \citep{lauritzen2001representing}, and the statistical process control problem \citep{cobb2021statistical}. For a discussion on the similarities and differences of influence diagrams and POMDPs, we refer to \citet{hansen2021integrated} who also presents a specialized algorithm that combines the variable elimination algorithm for influence diagrams and the value iteration algorithm for POMDPs.
    
     There is evidence that the RJT formulation outperforms the Decision Programming framework computationally when the aim is to optimize the CVaR risk measure or when the expected utility-maximizing decision strategy is to be determined in the presence of chance constraints \citep{herrala2025risk}. To the best of our knowledge, no studies have compared the RJT formulation with Decision Programming in computing expected utility-maximizing decision strategies more generally. On the other hand, as demonstrated by the numerical experiments in the present paper, the RJT formulation is inefficient for building and solving MILP models if the influence diagram yields a tree with large treewidth. This occurs, for instance, if the influence diagram lacks the periodic structure of POMDPs or contains nodes with a large number of possible states. See \citep{parmentier2020integer,herrala2025risk} for an in-depth discussion on the size of the resulting MILP model from the RJT formulation and its relationship to the treewidth. 
     
     Improvements in computational efficiency of solving influence diagrams whose nodes have large state spaces are needed due to the well-documented challenges of representing continuous events in influence diagrams \citep{bielza2010modeling}. These challenges have motivated, for example, the development of the variable elimination algorithm for Gaussian influence diagrams that have continuous random variables \citep{shachter1989gaussian,poland1993mixtures}; methods that use mixtures of truncated exponentials that are easy to integrate in closed form \citep{cobb2007influence,cobb2008decision}; and Monte Carlo sampling \citep{charnes2004multistage}. However, the resulting optimization models are either extremely challenging nonlinear (and possibly large) problems that can only be solved approximately \citep{charnes2004multistage,cobb2007influence,cobb2008decision}; or they require very restrictive assumptions on the structure of the influence diagram. For instance, \citet{shachter1989gaussian} assume that uncertainties are exogenous, as decisions cannot influence the probability distributions of continuous random variables. Due to these limitations, the standard approach is to discretize the continuous variables. In Bayesian networks and influence diagrams, this can be performed based on the expert judgment (see, e.g., \citet{keating2023using,mancuso2021optimal,de2025choosing}) or by using specialized statistical methods \citep{beuzen2018comparison}. Specifically, in large influence diagrams, such as \citet{keating2023using}, \citet{contasti2023balancing}, or \citet{li2024dynamic}, the limitations of the RJT formulation may force a coarser state discretization than desirable, which may erode the accuracy of the decision recommendations. 
     
     Against this backdrop, this paper develops the Reformulated Decision Programming (RDP) formulation, which helps solve influence diagrams that are difficult for the RJT formulation, i.e., those with a non-periodic structure and large state spaces for a comparatively large number of random variables and few decisions. Such diagrams are common in many practical applications, as evidenced by numerous reported examples \citep{basnet2023selecting, contasti2023balancing,keating2023using,khakzad2021optimal, li2024dynamic, neuvonen2025optimizing}. In this sense, the RDP formulation addresses a key limitation of the RJT formulation by making it possible to solve influence diagrams that are not tractable with earlier MILP formulations. It still retains the modeling flexibility of the MILP frameworks, such as the use of the CVaR risk measure in constraints or the objective function. The computational efficiency of this  formulation is demonstrated by solving variants of reported problems from the literature and comparing computation times with those for the RJT and Decision Programming formulations. In particular, the computational experiments show that the RDP formulation significantly outperforms its RJT counterpart in influence diagrams that cannot be cast as POMDPs and that have nodes with large state spaces. 
    
    The remainder of this paper is organized as follows. Section \ref{sec:Background} presents influence diagrams and the Decision Programming formulation. Section \ref{sec:Reformulation} develops the RDP formulation and discusses its computational aspects. Section \ref{sec:Results} reports computational experiments. Finally, Section \ref{sec:Conclusions} concludes by outlining directions for further research.  
    
\section{Background}\label{sec:Background}
\subsection{Influence Diagrams}

An influence diagram $G = (N, A)$ is a directed acyclic graph consisting of nodes $N = C \cup D \cup V$ and directed arcs $A$ connecting the nodes. These nodes correspond to different components of a decision problem. Chance nodes $C$ represent random variables, decision nodes $D$ represent choices among alternatives, and value nodes $V$ represent consequences. The arcs represent relationships between the nodes. For each node $n \in N$, $I(n)$ denotes the set of its parent nodes from which there is an arc to $n$. Following the terminology in \citet{salo2022decision}, we refer to this set $I(n)$ as the {\it information set} of $n$. We assume that value nodes do not belong to any information set. Specifically, arcs to decision nodes $d \in D$ represent the information on the basis of which the decision at node $d$ is made, for example, by having observed the states of chance and decision nodes $I(d)$. Arcs to chance nodes $c \in C$ represent probabilistic dependence so that $I(c)$ consists of those nodes that condition the probability distribution of the random variable at $c$. More strongly, the random variable at $c$ is independent of the non-descendants of $c$ given $I(c)$. Finally, arcs leading to value nodes $v \in V$ indicate those decisions $d \in D$ and random variables $c \in C$, which are the attributes of the utility function associated with $v \in V$.

Chance and decision nodes have discrete and finite sets of states $S_n$, $\forall n \in C \cup D$. At chance nodes $c \in C$, the set $S_c$ represents possible outcomes of the random variable and at decision nodes $d \in D$, the set $S_d$ represents decision alternatives. The joint state space of a set of nodes $M \subseteq C \cup D$ is the Cartesian product $S_M = \bigtimes_{m \in M}S_m$. Moreover, the lowercase notation $s_n \in S_n$ denotes a state, i.e., an element in the state space. Similarly, $s_M = (s_m)_{m \in M}$ denotes a combination of states of the nodes contained in $M$. For node $n \in N$, the information state $s_{I(n)} \in S_{I(n)}$ is the collection of states in the corresponding information set. 

The state of a chance node $s_c \in S_c$, $c \in C$, is a realization of a random variable that is characterized by a probability distribution $\mathbb{P}(s_c \mid s_{I(c)})$ that is conditional on the respective information state. These conditional probabilities are parameters in the numerical specification of the influence diagram. They can be estimated, for example, by using statistical methods or by eliciting expert statements \citep{bielza2010modeling,falconer2024eliciting,french2024whose}. 

At a decision node, the state $s_d \in S_d$, $d \in D$, represents a choice among the decision alternatives after observing $I(d)$. The selected alternative is determined by the local decision strategy $Z_d: S_{I(d)} \to S_d$. The strategy is represented with a binary variable $z(s_d \mid s_{I(d)}) \in \{0,1\}$, such that $z(s_d \mid s_{I(d)}) = 1$ if and only if the alternative $s_d$ is selected for the information state $s_{I(d)}$; otherwise, $z(s_d \mid s_{I(d)}) = 0$. The global decision strategy is denoted as $Z: \bigtimes_{d \in D}Z_d$. At each decision node, exactly one decision alternative per each information state is selected, i.e., 
\begin{equation}
\label{eq:feasible_dec_strat}
\sum_{s_d \in S_{d}}z(s_d \mid s_{I(d)}) = 1, \ \forall s_{I(d)} \in S_{I(d)}.
\end{equation}A decision strategy is \emph{feasible} if, in addition to \eqref{eq:feasible_dec_strat}, $Z$ fulfills other relevant constraints, such as budget limitations and logical dependencies.

The utilities of consequences associated with decisions and realizations of chance events are modeled at value nodes by utility functions $U_v: S_{I(v)} \rightarrow \mathbb{R}, v \in V$ that map each information state to a real number. Because utility functions are unique up to positive affine transformations \citep{savage1954foundations}, we assume without loss of generality that all utilities are strictly positive $U_v(s_{I(v)}) > 0, \forall v \in V, s_{I(v)} \in S_{I(v)}$. If there are several value nodes, assume that the aggregate utility $U: \bigtimes_{v \in V}S_{I(v)} : \mathbb{R}$ can be expressed as the sum over the node-specific utilities.

%In most cases, the solution of an influence diagram is a feasible decision strategy that maximizes the expected utility of the consequences. Typically, the decision strategy is not subject to other constraints than \eqref{eq:feasible_dec_strat}. Solution methods, such as variable elimination \citep{shachter1986evaluating} or single-policy-update algorithm \citep{lauritzen2001representing}, return one such strategy. To the best of our knowledge, they do not recover the full set of optima when several exist. In contrast, the MILP-based frameworks can be adapted to optimize measures other than expected utility, such as CVaR, and they can be easily adapted to find the full set of optimal strategies.

In most cases, an influence diagram is solved by finding a feasible decision strategy that maximizes expected utility, subject only to the constraint \eqref{eq:feasible_dec_strat}. Solution algorithms, such as those based on node removals and arc reversals \citep{shachter1986evaluating}, or single-policy-update \citep{lauritzen2001representing}, return one such strategy. However, during the solution process, they typically compute the expected utilities for other strategies which can (but do not necessarily) include all alternative optima. Such information is available in influence diagram solvers such as GeNIe \citep{bayesfusion2017genie}, whose default algorithm exhaustively computes the expected utilities of all feasible strategies \citep{cooper1988method}, and Hugin \citep{madsen2005hugin}, whose default algorithm iterates over global strategies until it finds one that cannot be improved by changing one of the local decision strategies in it \citep{lauritzen2001representing}. However, the algorithms in these two solvers are limited in that they either assume perfect recall \citep{bayesfusion2017genie} or do not guarantee convergence to a global optimum in some limited memory influence diagrams \citep{madsen2005hugin}. Exact algorithms for limited memory diagrams, such as multiple policy updating \citep{maua2012solving}, return a single optimal strategy but may fail to generate all alternative optimal strategies. In contrast, MILP-based frameworks are capable of finding all globally optimal strategies even without explicitly enumerating all strategies \citep[Section~5.4]{salo2022decision}. They can also be adapted to optimize objective measures other than expected utility, such as CVaR.

\subsection{Decision Programming}
\label{sec:dec_proc}

Decision Programming \citep{salo2022decision} (DP) is a framework for converting an influence diagram into an MILP model. The formulation below was proposed by \citet{hankimaa2025solving} to improve the original MILP model with valid inequalities and reformulations.

In the DP formulation, a \emph{path} $s = (s_n)_{n \in C \cup D}$ is a sequence of states for each chance and decision node. The set of all possible paths is denoted as $S = S_{C \cup D}$. A sequence of states $(s_m)_{m \in M}$ for a subset $M \subseteq C \cup D$ of nodes is a (path) segment. Probability $p(s) = \prod_{c \in C}\mathbb{P}(s_c \mid s_{{I(c)}})$ represents the product of conditional probabilities of the states of the chance nodes contained in path $s$. Denote $S^{>}:= \{s \in S \mid p(s) > 0 \}$ as the set of those paths for which this probability is strictly greater than zero. Similarly, the utility associated with a path is calculated as $U(s) = \sum_{v \in V}U(s_{I(v)})$. Both $p(s)$ and $U(s)$ are parameters that can be precalculated for each path $s \in S$.

An \emph{extension} of a path segment $s_M \in S_M, M \subsetneq C \cup D$ is defined as $E(s_M) := \{s' \in S \mid s'_M = s_M\}$, that is, the set of all paths that contain the segment $s_M$. Analogously, an extension over the paths with positive probability is denoted as $E^{>}(s_M) := E(s_M) \cap S^{>}$.

The variables in the DP formulation are decision strategy variables $z(s_d \mid s_{I(d)}) \in \{0,1\}, \forall d \in D, s_d \in S_{d}, s_{I(d)} \in S_{I(d)}$, and the \emph{path variables} $x(s) \in \{0,1\}, \forall s \in S^{>}$, defined as
\begin{equation*}
x(s) = 1 \Leftrightarrow z(s_d \mid s_{I(d)}) = 1, \forall d \in D.\end{equation*} 

The path variables indicate which paths are active given the selected decision strategy and therefore contribute to the expected utility. Specifically, a path $s \in S$ is active given a decision strategy $Z$, if $z(s_d \mid s_{I(d)}) = 1, \forall d \in D$, in which case $x(s) = 1$. On the other hand, if $z(s_d \mid s_{I(d)}) = 0$ for some $d \in D$, the path is not active given the decision strategy, and the path variable $x(s)$ is zero. Given $x(s)$, $\forall s \in S^{>}$, the expected utility of strategy $Z$ is 
\begin{equation}
\label{eq:eu}
    \sum_{s \in S^{>}}p(s)U(s)x(s).
\end{equation}
The optimal decision strategy that maximizes expected utility can be evaluated by solving the MILP model \eqref{eq:dp_objective}-\eqref{eq:dp_path_compat}
\begin{subequations}
\begin{align}
    \max_{z,x} &\sum_{s \in S^{>}}p(s)U(s)x(s) \label{eq:dp_objective}\\
    \text{s.t.} &\sum_{s_d \in S_{d}}z(s_d \mid s_{I(d)}) = 1,  ~&\forall d \in D, s_{I(d)} \in S_{I(d)}  \label{eq:dp_strategy1}\\ 
    & \sum_{s \in E^{>}(s_d,s_{I(d)})} x(s) \leq \Gamma(s_d \mid s_{I(d)}) z(s_d \mid s_{I(d)}),~&\forall d \in D, s_d \in S_d, s_{I(d)} \in S_{I(d)} \label{eq:dp_path_compat_strategy}\\
    & z(s_d \mid s_{I(d)}) \in \{0,1\},
     ~&\forall d \in D, s_d \in S_d, s_{I(d)} \in S_{I(d)}\label{eq:dp_strategy} \\
    & x(s) \in [0,1], ~&\forall s \in S^{>}. \label{eq:dp_path_compat}
\end{align}
\end{subequations}

Constraints \eqref{eq:dp_strategy1} are equivalent to \eqref{eq:feasible_dec_strat}. Constraints \eqref{eq:dp_path_compat_strategy} enforce that if $z(s_d \mid s_{I(d)}) = 0$, then $x(s)$ attains value 0 for all paths $s \in E^{>}(s_d, s_{I(d)})$. The value $\Gamma(s_d \mid s_{I(d)})$ acts as a big-M value so that if $z(s_d \mid s_{I(d)}) = 1$, the right hand side of \eqref{eq:dp_path_compat_strategy} allows all path variables $x(s)$ for paths such that $s \in E^{>}(s_d,s_{I(d)})$, to take value 1. For \eqref{eq:dp_path_compat_strategy} \citet{hankimaa2025solving} propose using $\Gamma(s_d \mid s_{I(d)})$ as
\begin{equation}
\Gamma(s_d\mid s_{I(d)}) = 
\min\left\{|E^{>}(s_d, s_{I(d)})|, \frac{|E(s_d,s_{I(d)})|}{\prod_{k \in D \setminus (\{d\} \cup I(d))}|S_k|}\right\},
\end{equation}

\noindent where $|E^{>}(s_d, s_{I(d)})|$ is the number of elements in the sum in the left hand side of Constraint \eqref{eq:dp_path_compat_strategy} and $\frac{|E(s_d,s_{I(d)})|}{\prod_{k \in D \setminus (\{d\} \cup I(d))}|S_k|}$ is an upper bound on how many path variables can be strictly positive in the sum on the left hand side of Constraint \eqref{eq:dp_path_compat_strategy}. For a more rigorous explanation of the parameter $\Gamma(s_d \mid s_{I(d)})$, see \citet{hankimaa2025solving}. Constraints \eqref{eq:dp_strategy} and \eqref{eq:dp_path_compat} define the domain of the variables.

Though the domain of path variables allows for fractional solutions in \eqref{eq:dp_path_compat}, the variables still find the optimal solution from the set $\{0,1\}$. This is because the upper bound on constraint \eqref{eq:dp_path_compat_strategy} is either 0 (if for any $d \in D, z(s_d \mid s_{I(d)}) = 0$) or $ \Gamma(s_d \mid s_{I(d)})$, which is greater than or equal to $1$. If $ \Gamma(s_d \mid s_{I(d)})$ is the upper bound in constraints \eqref{eq:dp_path_compat_strategy}, constraints \eqref{eq:dp_path_compat} impose a tighter upper bound for $x(s)$. Because $U(s) > 0, \forall s \in S$, the maximization of the objective function guides all possible $x(s)$ to attain the highest possible value, which is either $0$ or $1$. The use of continuous path variables is further discussed by \citet{hankimaa2025solving}, who show that continuous path variables significantly improve the computational efficiency of the DP formulation.

By design, the DP formulation finds a single policy that maximizes the objective function. If the interest is to find all optimal policies or the $k$ best policies or, in the case of multiple objectives, all Pareto optimal policies, the optimization model can be re-solved with constraints that render the already found decision strategies infeasible. These constraints, presented in \citep[Section 5.4]{salo2022decision}, can be readily added to the DP formulation.

%Influence diagram preprocessing techniques, such as removing barren nodes \citep{shachter1986evaluating} and removing non-requisite informational arcs (i.e. information that is not leveraged to make decisions) \citep{shacter1998bayes}, can and should be applied before formulating and solving the DP model. However, as the focus of the present paper is on MILP formulations, investigating these techniques is left as a topic for future research.
For improved computational performance, the influence diagram can be simplified before formulating the MILP model using standard preprocessing steps, such as removing barren nodes \citep{shachter1986evaluating} and removing non-requisite arcs leading to decision nodes from parents that are d-separated from the value nodes \citep{shacter1998bayes} given the other nodes in the information set of the decision node. We note that non-requisite arcs to decision nodes, identified with d-separation, may be requisite if the objective is to optimize CVaR or expected utility under chance constraints.

\section{Contributions}\label{sec:Reformulation}

\subsection{Reformulated Decision Programming model}
\label{sec:eu_models}

As \citet{salo2022decision} note, the DP formulation suffers from the curse of dimensionality. The computational studies in \cite{hankimaa2025solving} and \cite{herrala2025risk} further show that the solution times for the DP formulation grow exponentially with respect to the number of nodes in the influence diagram. The main reason for this is the exponential growth of the number of path variables as a function of the number of chance and decision nodes. The experiments in \citep{hankimaa2025solving} also demonstrate that the LP-relaxation of the DP formulation is weak, which significantly reduces the efficiency of MILP solvers. This motivates the development of an alternative formulation for the DP model with fewer variables and a tighter formulation.

Let $C_I = \{c \in C \mid \exists d \in D \text{ s.t. } c \in I(d)\}$ and let $O = D \cup C_I$. Thus, the set $O$ contains the decision nodes and chance nodes that inform at least one decision during the decision process. The set $O$ is called the \emph{observation set}. A path segment $s_O$ consisting of the states of nodes from the observation set is an \emph{observable segment}. The reformulation is based on the following finding concerning the optimal values for path variables in the DP formulation. The observable segment $s_O$ determines the values of path variables in the sense that if $x(s) = 1$ for some $s \in S^{>}$, then $x(s') = 1, \forall s' \in E^{>}(s_{O})$. Essentially, path variables for paths with identical observation segments share the same value at the optimal solution. This is due to the fact that $s_O$ contains all the states required to evaluate decision strategies, which ultimately determine whether $x(s) = 0$ or $x(s) = 1$. Proposition \ref{prop:suboptimal_solution} confirms the validity of the observation.
\begin{proposition}
\label{prop:suboptimal_solution}
    Assume that $(x,z)$ satisfies constraints \eqref{eq:dp_strategy1}-\eqref{eq:dp_path_compat}. Let $x^{*}$ be such that $x^{*}(s) = 1$ if $\max_{s' \in E^{>}(s_O)}x(s') > 0$ and $x^{*}(s) = 0$ otherwise. Then, $(x^{*},z)$ is feasible for the constraint set \eqref{eq:dp_strategy1}-\eqref{eq:dp_path_compat}. Moreover, if $\exists s \in S^{>}$ such that $x^{*}(s) \neq x(s)$, then $\sum_{s \in S^{>}}p(s)U(s)x(s) < \sum_{s \in S^{>}}p(s)U(s)x^{*}(s)$.
\end{proposition}
\begin{proof}
See Appendix \ref{sec:proofs}.
\end{proof}

As a corollary of Proposition \ref{prop:suboptimal_solution}, any feasible solution $(x,z)$ for \eqref{eq:dp_strategy1}-\eqref{eq:dp_path_compat} for which $\exists s, s' \in S^{>}$ such that $s_O = s'_O$ and $x(s) \neq x(s')$ is suboptimal. Define \emph{observation variables} $y(s_{O}) \in \{0,1\}, \forall s_{O} \in S_{O}$ based on the following relationship to path variables:
\begin{equation}
\label{eq:reduced}
    y(s_{O}) = 1 \Leftrightarrow x(s) = 1, \forall s \in E^{>}(s_{O}).
\end{equation}
Effectively, the observation variables $y(s_O)$ aggregate the path variables $x(s)$, for $s \in E^>(s_O)$ from the DP formulation. Consequently, when compared to path variables in the DP formulation, observation variables exclude solutions such that $x(s) \neq x(s')$, where $s, s' \in E^{>}(s_{O}),$ for some $s_O \in S_O$.

Let \begin{equation}
\label{eq:eut}
    \mathbb{E}_{U}(s_{O}) = \sum_{s \in E(s_{O})}p(s)U(s)
\end{equation} be the expected utility associated with an observable segment $s_O \in S_O$. Additionally, let $E_O(s_M) := \{s'_O \in S_O \mid s'_M = s_M\}$ be the \emph{observable extension} for an observable segment $M \subseteq O$. The expected utility of the decision strategy can be represented with the observation variables as $\sum_{s_{O} \in S_{O}}\mathbb{E}_{U}(s_{O})y(s_{O})$. Proposition \ref{prop:objectives_match} states that the sum $\sum_{s_{O} \in S_{O}}\mathbb{E}_{U}(s_{O})y(s_{O})$ is equal to the expected utility \eqref{eq:eu} expressed with the path variables in \eqref{eq:reduced}. 

\begin{proposition}
\label{prop:objectives_match}
    Assume that $y(s_O) \in \{0,1\}, \forall s_O \in S_O$, and $x$ is defined as in \eqref{eq:reduced} based on $y$. Then, 
    \begin{equation*}
    \sum_{s \in S_O}\mathbb{E}_{U}(s_O)y(s_O) = \sum_{s \in S^{>}}p(s)U(s)x(s). \end{equation*}
\end{proposition}

\begin{proof}
See Appendix \ref{sec:proofs}.
\end{proof}

Proposition \ref{prop:objectives_match} gives the following equivalent formulation.  
\begin{subequations}
\begin{align}
    \max_{z,y} &\sum_{s_{O} \in S_{O}}\mathbb{E}_{U}(s_{O})y(s_{O}) \label{eq:dpr_objective}\\
    \text{s.t.} &\sum_{s_d \in S_{d}}z(s_d \mid s_{I(d)}) = 1,  ~&\forall d \in D, s_{I(d)} \in S_{I(d)}  \label{eq:dpr_strategy1}\\ 
    & \sum_{s_{O} \in E_{O}(s_d,s_{I(d)})} y(s_{O}) \leq \Gamma(s_d \mid s_{I(d)}) z(s_d \mid s_{I(d)}), ~&\forall d \in D, s_d \in S_d, s_{I(d)} \in S_{I(d)} \label{eq:dpr_path_compat_strategy}\\
    & z(s_d \mid s_{I(d)}) \in \{0,1\},
     ~&\forall d \in D, s_d \in S_d, s_{I(d)} \in S_{I(d)}\label{eq:dpr_strategy} \\
    & y(s_{O}) \in [0,1], ~&\forall s_{O} \in S_{O}. \label{eq:dpr_path_compat}
\end{align}
\end{subequations}
In particular, the optimal solution to \eqref{eq:dpr_objective}-\eqref{eq:dpr_path_compat} is the expected utility maximizing decision strategy. Note that Proposition \ref{prop:objectives_match} provides the equivalence for binary $y(s_O)$ whereas the model \eqref{eq:dpr_objective}-\eqref{eq:dpr_path_compat} allows $y(s_O)$ to be a continuous variable on the interval $[0,1]$. However, with similar argumentation as in Section \ref{sec:dec_proc}, $y(s_O)$ finds the optimal solution from the set $\{0,1\}$ and the use of continuous $y(s_O)$ is warranted. In addition, because the number of observation variables cannot exceed the number of path variables, the coefficient $\Gamma(s_d \mid s_{I(d)})$ remains a valid big-M value in Constraint \eqref{eq:dpr_path_compat_strategy}. Propositions \ref{prop:equivalence} and \ref{prop:optimality} show that the formulation \eqref{eq:dpr_objective}-\eqref{eq:dpr_path_compat} gives the optimal value of the DP formulation.

\begin{proposition}
    \label{prop:equivalence}
    Assume that ($y$, $z$) satisfies constraints \eqref{eq:dpr_strategy1}-\eqref{eq:dpr_path_compat} and $y(s_O) \in \{0,1\}, \forall s_O \in S_O$. Let $x^{*}(s) = y^{*}(s_O), \forall s \in S^{>}$. Then, $(x,z)$ satisfies the constraints \eqref{eq:dp_strategy1}-\eqref{eq:dp_path_compat}.
\end{proposition}

\begin{proof}
See Appendix \ref{sec:proofs}.
\end{proof}

\begin{proposition}
\label{prop:optimality}
Assume that ($y^{*}$, $z^{*}$) is the optimal solution for \eqref{eq:dpr_objective}-\eqref{eq:dpr_path_compat} and $x^{*}(s) = y^{*}(s_O), \forall s \in S^{>}$. Then, $(x^{*},z^{*})$ is the optimal solution for \eqref{eq:dp_objective}-\eqref{eq:dp_path_compat}.
\end{proposition}

\begin{proof}
   See Appendix \ref{sec:proofs}.
    \end{proof}

Furthermore, formulation \eqref{eq:dpr_objective}-\eqref{eq:dpr_path_compat} can be strengthened with a valid inequality by recognizing that for each path segment containing all chance nodes $s_C$, there exists at most one path $s$ containing $s_C$ such that $x(s) = 1$. In \eqref{eq:dpr_objective}-\eqref{eq:dpr_path_compat}, the equivalent statement is that for each $s_{C_I}$, there exists at most one $s_O$, containing $s_{C_I}$ such that $y(s_O) = 1$. Proposition \ref{prop:valid_inequality} uses this to generate a valid inequality for \eqref{eq:dpr_objective}-\eqref{eq:dpr_path_compat}.

\begin{proposition}
\label{prop:valid_inequality}
    Assume that $(y, z)$ satisfies constraints \eqref{eq:dpr_strategy1}-\eqref{eq:dpr_path_compat}. Then, $\sum_{s_{O} \in E_{O}(s_{C_I})}y(s_{O}) \leq 1, \forall s_{C_I} \in S_{C_I}$.
\end{proposition}

\begin{proof}
 See Appendix \ref{sec:proofs}.
\end{proof}

Therefore, $\sum_{s_{O} \in E_{O}(s_{C_I})}y(s_{O}) \leq 1, \forall s_{C_I} \in S_{C_I}$, can be added to \eqref{eq:dpr_objective}-\eqref{eq:dpr_path_compat} for an improved LP relaxation. The final form of the RDP formulation is given in \eqref{eq:dpr2_objective}-\eqref{eq:dpr2_path_compat}.
\begin{subequations}
\begin{align}
    \max_{z,y} &\sum_{s_{O} \in S_{O}}\mathbb{E}_{U}(s_{O})y(s_{O}) \label{eq:dpr2_objective}\\
    \text{s.t.} &\sum_{s_d \in S_{d}}z(s_d \mid s_{I(d)}) = 1,  ~&\forall d \in D, s_{I(d)} \in S_{I(d)}  \label{eq:dpr2_strategy1}\\ 
    & \sum_{s_{O} \in E_{O}(s_d,s_{I(d)})} y(s_{O}) \leq \Gamma(s_d \mid s_{I(d)}) z(s_d \mid s_{I(d)}), ~&\forall d \in D, s_d \in S_d, s_{I(d)} \in S_{I(d)} \label{eq:dpr2_path_compat_strategy}\\
    & \sum_{s_{O} \in E_{O}(s_{C_I})} y(s_{O}) \leq 1, ~&\forall s_{C_I} \in S_{C_I} \label{eq:dpr2_valid_ineq}\\
    & z(s_d \mid s_{I(d)}) \in \{0,1\},
     ~&\forall d \in D, s_d \in S_d, s_{I(d)} \in S_{I(d)}\label{eq:dpr2_strategy} \\
    & y(s_{O}) \in [0,1], ~&\forall s_{O} \in S_{O}. \label{eq:dpr2_path_compat}
\end{align}
\end{subequations}

\begin{rmk}
    The RDP formulation can be preprocessed as the DP formulation by defining observation variables only for the observable segments with an expected utility contribution that is strictly positive $S^{>}_O = \{s_O \in S_O \mid \mathbb{E}_{U}(s_O) > 0\}$.
\end{rmk}

\subsection{Alternative objective functions and risk constraints}

The DP formulation can be adapted to model a range of risk preferences to optimize risk measures like the CVaR, which is the expected value of consequences in the worst-case $\alpha$-tail of the distribution where $\alpha$ is a probability threshold parameter \citep{rockafellar2000optimization}. Similarly, it is possible to enforce chance constraints \citep{charnes1959chance} or logical constraints. Note that in the models presented in Section \ref{sec:eu_models}, $U(s)$ can be defined separately for every path $s \in S$, which offers a flexible way to encode risk into the utility function. In that case, maximizing expected utility with the models presented in Section \ref{sec:eu_models} leads to a decision strategy that is consistent with the risk tolerance profile of the decision maker. This section demonstrates how the RDP formulation \eqref{eq:dpr2_objective}-\eqref{eq:dpr2_path_compat} can, with a suitable reformulation, be generalized to include CVaR in the objective function or to account for risk constraints. 

\subsubsection{Conditional value-at-risk}
\label{sec:cvar}

The CVaR formulation in \citet{salo2022decision} takes advantage of the fact that the probability distribution over consequences can be represented as a function of the decision variables. In the DP formulation, this is relatively straightforward as each path $s \in S$ is associated with a utility $U(s)$. Thus, the probability of attaining utility value $u$ can be represented as $q(u) = \sum_{\{s \in S \mid U(s) = u\}}p(s)x(s)$. 

To represent the probability distribution of consequences as variables in the RDP formulation, define the \emph{utility extension} $E^u_U(s_O) = \{s \in E(s_O) \mid U(s) = u\}$. The utility extension for the observable segment $s_O$ is the set of paths that contain $s_O$ and whose utility is $u$. The probability of attaining this utility $u$ is 
\begin{equation*}
q(u) = \sum_{s_O \in S_O}y(s_O)\left(\sum_{s \in E^u_U(s_O)}p(s)\right).
\end{equation*}
Additionally, let us define the set $\mathcal{U} = \{u \in \mathbb{R} \mid \exists s \in S \text{ s.t. } U(s) = u\}$. As in \cite{salo2022decision}, constraints \eqref{eq:cvar-1}-\eqref{eq:cvar-11} are posed. 
\begin{subequations}
\begin{align}
    &\eta - u \le M \lambda(u), \ & \forall u \in \mathcal{U} \label{eq:cvar-1}\\ 
    &\eta - u \ge (M+\epsilon) \lambda(u) - M, \ & \forall u \in \mathcal{U} \label{eq:cvar-2}\\ 
    &\eta - u \le (M+\epsilon)\overline{\lambda}(u) - \epsilon, \ & \forall u \in \mathcal{U} \label{eq:cvar-3}\\ 
    &\eta - u \ge M (\overline{\lambda}(u)-1), \ & \forall u \in \mathcal{U} \label{eq:cvar-4}\\ 
    &\overline{\rho}(u) \le \overline{\lambda}(u), \ & \forall u \in \mathcal{U} \label{eq:cvar-5}\\
    &q(u) - (1-\lambda(u)) \le \rho(u) \le \lambda(u) , \ & \forall u \in \mathcal{U} \label{eq:cvar-6}\\
    &\rho(u) \le \overline{\rho}(u) \le q(u), \ & \forall u \in \mathcal{U} \label{eq:cvar-7}\\ 
    &\sum_{u \in \mathcal{U}} \overline{\rho}(u) = \alpha \label{eq:cvar-8}\\
    &\lambda(u), \overline{\lambda}(u) \in \{0,1\}, \ & \forall u \in \mathcal{U} \label{eq:cvar-9}\\ 
    &\rho(u), \overline{\rho}(u) \in [0,1], \ & \forall u \in \mathcal{U} \label{eq:cvar-10}\\
    &\eta \in \mathbb{R}. \label{eq:cvar-11}
\end{align}
\end{subequations}

In \eqref{eq:cvar-1}-\eqref{eq:cvar-11}, $\alpha$ is the threshold parameter and $\epsilon$ is a small constant used to model strict inequalities. \citet{salo2022decision} use $\epsilon = \frac{1}{2}\min\{|u_1 - u_2| \mid u_1, u_2 \in \mathcal{U}, u_1 \neq u_2\}$, which is the smallest positive utility difference between two paths, and $M$ is a big-M value. For a more detailed explanation of \eqref{eq:cvar-1}-\eqref{eq:cvar-11}, see \citep{salo2022decision,herrala2025risk}.

When Constraints \eqref{eq:cvar-1}-\eqref{eq:cvar-11} are added to the RDP formulation, CVaR is given by 
\begin{equation}
\label{eq:cvar}    
\frac{1}{\alpha}\sum_{u \in \mathcal{U}}\overline{\rho}(u)u.
\end{equation}
This expression can be used as an alternative objective function instead of the expected utility or as a constraint in the optimization model. Specifically, maximizing \eqref{eq:cvar} will maximize the expected utilities in the worst $\alpha$-tail of the distribution, which is equivalent to minimizing the risks associated with the worst $\alpha$-tail of the utility distribution. 

The CVaR formulation \eqref{eq:cvar-1}-\eqref{eq:cvar-11} assumes that 
\begin{equation}
\label{eq:condition_for_cvar}
        \forall s_O \in S_O \mid z(s_d \mid s_{I(d)}) = 1, \forall d \in D \Leftrightarrow y(s_O) = 1.%, %\forall s_O \in S_O
\end{equation}
That is, if all decisions in an observable segment $s_O \in S_O$ are compatible with the decision strategy, then $y(s_O) = 1$. In the RDP formulation, this assumption is fulfilled implicitly by maximizing the expected utility. However, this assumption does not hold for more general objective functions. To see why, consider that in \eqref{eq:cvar-1}-\eqref{eq:cvar-11}, the variable $\overline{\rho}(u)$ represents the \emph{worst-case} $\alpha$-tail distribution of the utilities. This is done by enforcing a lower bound on $\overline{\rho}(u)$ in Constraints \eqref{eq:cvar-7} using another variable $\rho(u)$ that is bounded from below with expression $q(u)$, which in turn depends on $y(s_O)$. Due to Constraint \eqref{eq:cvar-8}, the lower bound $q(u)$ forces $\overline{\rho}(u)$ to values that are strictly positive in the worst-case tail of the probability distribution of the attainable utilities. In terms of the objective function \eqref{eq:cvar}, a better solution would be to set the lower bounds to zero for small utility values and seek to exhaust Constraint \eqref{eq:cvar-8} by using paths with large utilities. Therefore, if \eqref{eq:condition_for_cvar} is not enforced, the model can simply set $y(s_O) = 0$ for compatible observable segments with small utility values, which sets the lower bound in Constraint \eqref{eq:cvar-6} to zero and allows $\overline{\rho}(u)$ to seek nonzero values for large utilities. If \eqref{eq:condition_for_cvar} holds, then the lower bounds work as intended and $\overline{\rho}(u)$ represents the worst-case $\alpha$-tail distribution. Therefore, \eqref{eq:condition_for_cvar} is imposed by a constraint. This is done by enforcing that Constraints \eqref{eq:dpr2_valid_ineq} hold as an equality.  Proposition \ref{prop:condition_holds} shows that enforcing Constraints \eqref{eq:dpr2_valid_ineq} as equalities enforces \eqref{eq:condition_for_cvar}.

\begin{proposition}
\label{prop:condition_holds}
    Assume that $(y,z)$ is a feasible solution for \eqref{eq:dpr2_objective}-\eqref{eq:dpr2_path_compat} where \eqref{eq:dpr2_valid_ineq} holds with equality. Then, \eqref{eq:condition_for_cvar} holds.
\end{proposition}
\begin{proof}
    See Appendix \ref{sec:proofs}
\end{proof}

As a result, the CVaR maximizing decision strategy can be found by maximizing \eqref{eq:cvar} subject to Constraints \eqref{eq:dpr2_strategy1}-\eqref{eq:cvar-11} and enforcing that Constraints \eqref{eq:dpr2_valid_ineq} hold as equalities. This formulation is referred to as RDP-CVaR. The full model for maximizing CVaR is in Appendix \ref{sec:full_cvar_model}.

\begin{rmk}
    The CVaR-model cannot be preprocessed like the expected utility-maximizing model \eqref{eq:dpr2_objective}-\eqref{eq:dpr2_path_compat} by excluding observable segment variables $y(s_O)$ whose expected utility contribution is zero. Enforcing Constraints \eqref{eq:dpr2_valid_ineq} as equalities will render the model infeasible if an observable segment $s_O$ for which $\mathbb{E}_{U}(s_O) = 0$ is such that $z(s_d \mid s_{I(d)}) = 1, \forall d \in D$. Then, $y(s_O)$ would not be included in the model and some other path $s'_O \in E_O(s_{C_I})$ would be required to attain value $y(s'_O) = 1$ to fulfill Constraints \eqref{eq:dpr2_valid_ineq} as an equality. Based on Proposition \ref{prop:valid_inequality}, this is not possible and the model would be rendered infeasible. 
\end{rmk}

\subsubsection{Chance constraints}

Chance constraints can be introduced to the optimization model to ensure that the probability of a state or a collection of states does not exceed a given probability threshold. In the DP formulation, chance constraints are relatively simple to impose as the set of paths $S$ encompasses all possible state combinations of chance and decision nodes. Let $T \subseteq C \cup D$ be the set of nodes and let $\mathcal{S}_T \subseteq S_T$ be the set of state combinations that are subject to a chance constraint. Let $b_T \in [0,1]$ be a probability threshold. Chance constraints can be added to the DP formulation by introducing the constraint 
\begin{equation}
    \sum_{s \in E(s_T), s_{T} \in \mathcal{S}_T}p(s)x(s) \leq b_T.
\end{equation}
For the RDP formulation, the corresponding chance constraint is
\begin{equation}
\label{eq:chance_constraints}
    \sum_{s_T \in \mathcal{S}_T}\sum_{s_O \in E_O(s_{O \cap T})}y(s_O)\left(\sum_{s \in E(s_T)}p(s)\right) \leq b_T.
\end{equation}

The set $\mathcal{S}_T$ and the threshold $b_T$ can be adjusted to enable other types of constraints. As an example, constraints preventing budget violations should be enforced by selecting $T = \cup_{v \in \mathcal{V}}I(v)$, where $\mathcal{V} \subseteq V$ is the set of value nodes from which budget consumption is evaluated and letting $\mathcal{S}_T$ contain the state combinations for which the budget is violated. Then, enforcing Constraint \eqref{eq:chance_constraints} with $b_T = 0$ gives a decision strategy that does not exceed the budget. Similarly, constraints enforcing logical requirements (for example, equipment that must be purchased in order to be utilized) can be enforced by selecting $T$ as the nodes from which logical requirements are evaluated, and $\mathcal{S}_T$ as the joint states that violate the logical requirement. Enforcing Constraint \eqref{eq:chance_constraints} with $b_T = 0$ leads to a decision strategy that does not violate the logical requirements. Note that when enforcing chance constraints, \eqref{eq:dpr2_valid_ineq} must be enforced as equalities for the same reason as outlined in Section \ref{sec:cvar}.

\subsection{Computational considerations}

The RDP formulation improves the DP formulation as it has fewer variables and a tighter formulation that eliminates suboptimal solutions. The relative performance improvement of the RDP formulation compared to the DP formulation depends on the structure of the influence diagram. If all nodes are observed so that $O = C \cup D$, then the RDP formulation only improves upon the DP formulation by having the additional valid inequality \eqref{eq:dpr2_valid_ineq}.

A key tractability bottleneck in the DP formulation is the set of paths $S$, which grows exponentially with respect to the number of nodes, and their associated number of states, in the influence diagram. The RDP formulation still scales exponentially with respect to the size of the influence diagram as the number of observation variables grows exponentially with respect to the number of nodes in the observation set. The total number of constraints (excluding variable domain constraints) in the RDP model is $\mathcal{N}^c_{RDP} = |S_{C_I}| + \sum_{d \in D}(|S_d|+1)|S_{I(d)}|$, whereas in the DP model it is $\mathcal{N}^c_{DP} = \sum_{d \in D}(|S_d|+1)|S_{I(d)}|$. Hence, the number of constraints in the RDP formulation is always larger than in the DP formulation due to the valid inequality \eqref{eq:dpr2_valid_ineq}. Similarly, the number of decision variables in the RDP model is given as $\mathcal{N}^v_{RDP} = |S_{O}| + \sum_{d \in D}|S_d||S_{I(d)}|$ and in the DP model as $\mathcal{N}^v_{DP} = |S| + \sum_{d \in D}|S_d||S_{I(d)}|$. The reduced number of decision variables in the RDP framework is thus the difference between $|S|$ and $|S_O|$. Moreover, as $|S|$ grows exponentially with respect to the number of chance and decision nodes in the influence diagram and $|S_O|$ grows exponentially with respect to the number of nodes in $O$, there does not exist a limit on the number of constraints and variables that are needed to formulate the Decision Programming-based models. Although the RDP formulation does not associate a decision variable with each path, the full set of paths still appears in the objective function. To see this, notice that the sum $\sum_{s_{O} \in S_{O}}\mathbb{E}_{U}(s_{O})y(s_{O}) = \sum_{s_{O} \in S_{O}} \sum_{s \in E(s_O)}p(s)U(s)y(s_O)$ contains $|S|$ elements. On the other hand, $\mathbb{E}_{U}(s_{O})$ is a parameter that can be precalculated for each $s_{O}$, indicating the opportunity to leverage parallel computing to reduce the preprocessing time.

Compared to the RDP model, the RJT model scales differently. The number of variables and constraints in the RJT model grows exponentially with respect to the treewidth of the underlying junction tree, which does not directly depend on the number of nodes in the influence diagram. Additionally, the treewidth is affected by the number of arcs in the influence diagram. In contrast, the number of observation variables in the RDP model is independent of the number of arcs in the influence diagram. A more detailed discussion on the size and scaling of the RJT model can be found in \citet[Supplementary Material]{herrala2025risk}.

\section{Computational Experiments}\label{sec:Results}

The computational performance is assessed by solving variants of reported decision problems from the literature with the RJT, DP, and RDP formulations. The optimization models were implemented using Julia v1.10.3 with JuMP v1.23.0 \citep{Lubin2023} and solved with the Gurobi solver v11.0.2. For each problem type, two sets of computational experiments are performed. The first set uses the default Gurobi parameters to give insight into how efficiently an optimal decision strategy can be found. In the second set of experiments, Gurobi parameters that disable heuristics and model preprocessing are used to assess the effectiveness and characteristics of the MILP formulations. The used Gurobi parameter values for the second set of experiments are reported in Table \ref{tab:solver_parameters}. 

The RJT-models and the DP formulation were constructed using the Julia package DecisionProgramming.jl v2.0.1. The experiments were conducted using an Intel E5-2680 CPU at 2.5 GHz with 16 GB of RAM, provided by the Aalto University School of Science's ``Science-IT'' project. The code and data for the numerical experiments in this section are available at \href{https://github.com/toubinaattori/Decision-Programming-new-formulation}{https://github.com/toubinaattori/Decision-Programming-new-formulation}. In addition to these computational experiments, three additional computational experiments are presented in the Supplementary material, which is available \href{https://github.com/toubinaattori/Decision-Programming-new-formulation/blob/main/Supplementary_material.pdf}{here}. Two of the additional experiments are variations of the experiments presented in this Section, namely the turbine inspection and maintenance problem discussed in Section \ref{sec:turbine} and the extended oil wildcatter problem discussed in Section \ref{sec:oil}. The third additional example solves a variation of the classical pig farm problem from \citep{lauritzen2001representing}. 

\begin{table}[h!]
\caption{Solver parameter settings in the experiments.}\label{tab:solver_parameters}
\begin{tabular}{@{}l@{\quad}r@{}@{\quad}l@{}}
\toprule
Parameter              & Value & Description    \\ \midrule
TimeLimit    & 3600  & Limit total runtime (1h) \\ 
Threads    & 1  & Thread count \\ 
Presolve    & 0  & Disable presolve \\ 
Heuristics    & 0  & Minimize time spent in MIP heuristics \\ 
NodeMethod    & 1  & Use dual simplex for MIP node relaxations \\ 
Cuts    & 1  & Moderate cut generation \\ 
MIPFocus    & 2  & Focus on proving optimality 
\\ \bottomrule
\end{tabular}
\end{table}

\subsection{Assessing the optimization models}

%A starting point for solving an influence diagram with a MILP formulation is an influence diagram $G = (N,A)$ with discretized state spaces for the decision and chance nodes, conditional probability distributions for chance nodes, and numerical values of the utility functions at value nodes for the consequences. 

Depending on the framework, an influence diagram is preprocessed differently into an MILP formulation. To build the RJT model, the influence diagram is first converted to an RJT \citep{parmentier2020integer}. Then, the MILP formulation is created with the specified numerical parameters. To build the RDP model, parameters $\mathbb{E}_U(s_O)$ in \eqref{eq:eut} are calculated for each observable segment $s_O \in S_O$ and the corresponding MILP formulation in \eqref{eq:dpr2_objective}-\eqref{eq:dpr2_path_compat} is created. Similarly, to build the DP model, parameters $p(s)$ and $U(s)$ are calculated for each $s \in S$. The MILP formulation is created with the specified and calculated parameters. The \textit{preprocessing and model construction time} is denoted by $\tau_p$ and the \textit{solution time} is denoted as $\tau_s$. 

In practice, the most important measure of computational efficiency is the \emph{total time}, denoted as $\tau_t = \tau_p + \tau_s$, which is the time needed to formulate and solve the MILP model based on an influence diagram. For compactness, only the total and solution times are reported, as $\tau_p$ can be recovered from the reported numerical values as $\tau_p = \tau_t - \tau_s$. If a total time limit $\tau_t^{max} = 3600$ seconds is reached, either the best feasible solution or that no feasible solution was found within the time limit is reported. The average total time is denoted as $\overline{\tau}_t$ and the average solution time as $\overline{\tau}_s$, which are both calculated across the sample of randomly generated instances for each problem. In cases where not all generated instances were solved to optimality within the time limit, $\overline{\tau}_t$ and $\overline{\tau}_s$ are the average times of those instances that were solved to optimality. In addition, the number of instances solved to optimality and the number of instances for which a feasible solution was found are reported as $\# O/F$, where a reported result $12/18$ means that 12 optimal solutions and 18 feasible solutions were found. A reported result ``T'' denotes that no optimal solutions were found within the time limit. In some cases, generating and solving the MILP model required more memory than the allocated 16 GB. For these cases ``OoM" denotes that the cause of termination was that the processor was out of memory. The average gap to the tightest relaxation is denoted as $RR \%$. Specifically, the root node LP relaxations are compared against the tightest relaxation (with RR\% value equal to zero). Finally, the sizes of the optimization models are analyzed by reporting the number of constraints excluding the variable domain constraints (\#c) and the number of variables (\#v) in the resulting optimization models. To avoid repetition, only the model sizes with instances that use the Gurobi parameters in Table \ref{tab:solver_parameters} are reported.

\subsection{Turbine inspection and maintenance problem} \label{sec:turbine}

First, consider a variant of a turbine system inspection and maintenance problem in \citet{mancuso2021optimal}, representing a turbine and an associated sensor that are subject to wear and tear, which negatively affects the performance (i.e., the flow) of the turbine. Monitoring the sensor and the turbine condition gives estimates, based on which a decision to carry out a more thorough inspection can be made to obtain more accurate information about the true condition. After possible inspections, the system can be maintained to increase the turbine flow. This problem is an example of an influence diagram that lacks a periodic structure and is, as the numerical results corroborate, computationally demanding for the RJT framework. The nodes of the problem are listed in Table \ref{tab:nodes_turbine} and the influence diagram of the problem is in Figure \ref{fig:Turbine}.

\begin{table}
\caption
{Nodes of the turbine inspection and maintenance problem.} \label{tab:nodes_turbine}
\begin{tabular}{@{}l|@{\quad}c|@{}@{\quad}c|@{}@{\quad}c}
\toprule
Node              & Type &  States  & Explanation    \\ 
\midrule
$W$    & Chance node   & [0,100] & Environment conditions          \\ 
$FH$    & Chance node   & [0,100] & Fired hours          \\ 
$SS$    & Chance node & [0,100] & True condition of the sensor \\ 
$TS$    & Chance node & [0,100] & True condition of the turbine \\ 
$SE$    & Chance node & [0,100] & An estimate of the sensor condition \\ 
$TE$    & Chance node & [0,100] & An estimate of the turbine condition\\ 
$IN$    & Decision node & none, sensor check, turbine check & Decision to inspect the sensor and turbine \\ 
$SR$    & Chance node & $[0,100]$  & Result of sensor inspection \\ 
$TR$    & Chance node & $[0,100]$ & Result of turbine inspection \\ 
$M$    & Decision node & none, level 1, level 2 & Maintenance decision \\ 
$TF$    & Chance node & [0,100] & Turbine flow after possible maintenance \\ 
$U$    & Value node & - & Rewards minus the costs \\
\bottomrule
\end{tabular}
\end{table}

\begin{figure}[pos=htbp]
\centering 
\begin{tikzpicture}
    [decision/.style={color = black,fill=white!80, draw, minimum size=2.5em, inner sep=2pt}, 
    chance/.style={circle,color = black, fill=white!80, draw, minimum size=2.5em, inner sep=2pt},
    value/.style={diamond,color = black, fill=white!60, draw, minimum size=2.5em, inner sep=2pt},
    scale=1.5]
     %\draw[step=1cm,gray,very thin] (0,0) grid (4,3);
     \node[chance] (tl) at (3, 6)      {$W$};
     \node[chance] (fh) at (1, 6)      {$FH$};
     \node[chance] (ss) at (1, 5)      {$SS$};
     \node[chance] (ts) at (3, 5)      {$TS$};
     \node[chance] (se) at (1, 4)      {$SE$};
     \node[chance] (te) at (3, 4)      {$TE$};
     \node[decision] (i) at (2, 3)      {$IN$};
     \node[chance] (sr) at (1, 2)      {$SR$};
     \node[chance] (tr) at (3, 2)      {$TR$};
     \node[decision] (m) at (2, 1)      {$M$};
     \node[chance] (tf) at (2, 0)      {$TF$};
     \node[value] (u) at (2, -1)      {$U$};
     \draw[->, thick] (fh) -- (ss);
     \draw[->, thick] (fh) to[bend left = 10] (i);
     \draw[->, thick] (fh) to[bend right = 70] (m);
     \draw[->, thick] (fh) -- (ts);
     \draw[->, thick] (tl) -- (ts);
     \draw[->, thick] (tl) to[bend left = 50] (tf);
     \draw[->, thick] (ss) -- (se);
     \draw[->, thick] (ts) -- (te);
     \draw[->, thick] (se) -- (i);
     \draw[->, thick] (ts) -- (se);
     \draw[->, thick] (se) -- (te);
     \draw[->, thick] (ss) to[bend right = 30] (sr);
     \draw[->, thick] (ts) to[bend left = 30] (tr);
     \draw[->, thick] (ts) to[bend left = 40] (tf);
     \draw[->, thick] (te) -- (i);
     \draw[->, thick] (se) -- (sr);
     \draw[->, thick] (te) -- (tr);
     \draw[->, thick] (i) -- (sr);
     \draw[->, thick] (i) -- (tr);
     \draw[->, thick] (sr) -- (m);
     \draw[->, thick] (sr) -- (tr);
     \draw[->, thick] (tr) -- (m);
      \draw[->, thick] (m) -- (tf);
      \draw[->, thick] (m) to[bend right = 35] (u);
      \draw[->, thick] (i) to[bend left = 30] (u);
      \draw[->, thick] (tf) -- (u);
\end{tikzpicture}
%\vspace{-0.5cm}
\caption{Influence diagram of the turbine inspection and maintenance problem.} \label{fig:Turbine}
\vspace{-0.5cm}
\end{figure}
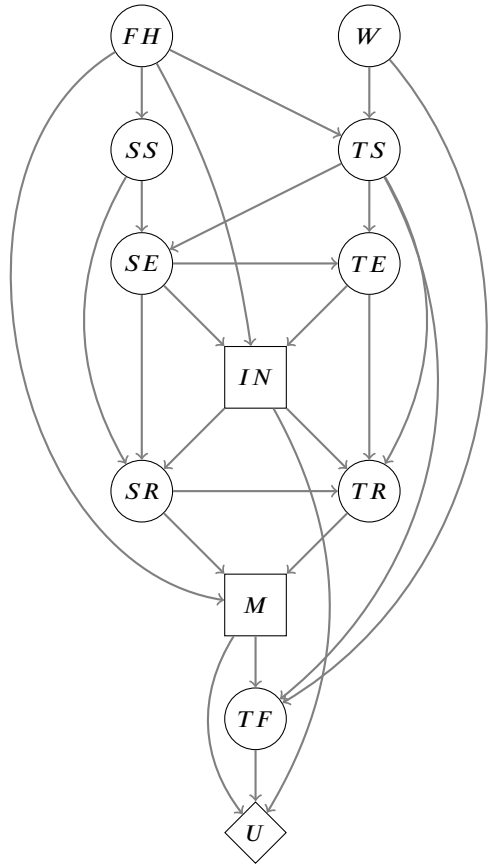

Compared to the influence diagram used in \citet{mancuso2021optimal}, the influence diagram is modified in three ways. In \citet{mancuso2021optimal}, the node $FH$ only has one state, which consequently is known at every decision. Multiple instances of the problem are then solved with a different state for node $FH$. Instead, this experiment solves a single influence diagram, where $FH$ has many states. Since the cumulative fired hours is assumed to be known at each decision node, the corresponding edges $(FH,IN)$ and $(FH,M)$ are added to the diagram. In addition, compared to \citet{mancuso2021optimal}, there is an additional chance node $W$ that represents exogenous conditions that affect the state and flow of the turbine, thus making the setup more realistic.

% TODO: Comment out the following two paragraphs paragraph if we decide to stick with the existing version

There are three inspection alternatives (no inspection, sensor check, and turbine check) and three maintenance alternatives (no maintenance, level-1 maintenance, and level-2 maintenance). All chance nodes are modeled with continuous random variables. In Table \ref{tab:nodes_turbine}, state $0$ represents the smallest possible state for the given node. For example, $s_{FH} = 0$ represents a brand new turbine, and $s_{SS} = 0$ represents the situation where the sensor is not operational. State $100$, on the other hand, represents the largest possible state. For example, $s_{FH} = 100$ represents the maximum duration that the turbine can run before maintenance must be conducted, and $s_{SS} = 100$ represents the situation where the sensor is in pristine condition. 

An influence diagram representation of the turbine inspection and maintenance problem requires that the state spaces of chance nodes are represented as a discrete set of states. The performance of the MILP models is tested by solving randomly generated instances of the turbine inspection and maintenance problem when representing the state spaces of chance nodes by $K$ equally wide intervals in their domain $[0,100]$. Utility function values and probability distributions are generated in a way that follows intuition while incorporating sufficient randomness to ensure a heterogeneous set of instances. For example, the probability of the sensor being in pristine condition should be high if the system has been running for a short time. Together, these design choices ensure that the resulting instances are realistic, yet varied enough to give meaningful insight into the effectiveness of the developed formulations.

Unconditional probability distributions for $W$ and $FH$ are created as random nonnegative vectors of length $K$ that sum to one. Conditional probability distributions are constructed as follows. For a chance node $c \in C$ and state $s_c \in S_c$ A weight vector $w \in \mathbb{R}_{+}^K$ is defined by \begin{equation}
\label{eq:probability_generation_formula}
w_{s_c} = e^{-r|s_c - \bar s|} (0.5 + \xi_i),\end{equation} where, rather than using the true state values, the states are indexed as $s_c \in \{1,\ldots,K\}$ and $w$ is normalized to sum to one, giving $\mathbb{P}(s_{c} \mid s_{I(c)}) = w_{s_c}$. The term $e^{-r|s_c - \bar s|}$ represents the shape of the probability distribution centered at the most probable state $\bar s$ given the information state $s_{I(c)}$. For example, if the turbine is brand new (indexed state $s_{FH} = 1$), the most probable sensor state is pristine (indexed state $s_{SE} = K)$. The parameter $r$ controls the sharpness of the distribution; larger values of $r$ concentrate probability mass more tightly around $\bar s$. Randomness is introduced by drawing $\xi_i \sim$ Uni$(0,1)$, where Uni$(0,1)$ represents the uniform distribution over the interval $[0,1]$. Appendix \ref{sec:used_parameters} specifies the used parameter values and the justification for them. Inspection costs are generated by drawing two numbers from Uni$(10,100)$ and assigning the smaller cost to the sensor check and the higher cost to the turbine check. Similarly, maintenance costs are generated by drawing two numbers from Uni$(500,3000)$ and assigning the lower cost to level-1 maintenance and the higher cost to level-2 maintenance. The rewards of the turbine flow are generated by drawing $K$ values from  Uni$(100,10000)$ and sorting them so that the smallest reward is associated with the smallest turbine flow state.

The computational performance of the formulations is tested by generating and solving 50 random instances of the turbine inspection and maintenance problem with discretizations of size $K = 2,\dots,6$ by using the RDP, RJT, and DP formulations. The results of the experiments are presented in Tables \ref{tab:turbine} (obtained with the Gurobi solver parameters shown in Table \ref{tab:solver_parameters}), and \ref{tab:turbine_gurobi} (obtained with the default Gurobi solver parameters). The results show that the RDP formulation outperforms the RJT and DP formulations as the optimal solution for all generated instances is found much faster than with other formulations. The main driver behind the results is the compact size of the RDP formulation compared to the other formulations. In instances of size $K = 6$, the RDP formulation has fewer than 10 000 constraints and fewer than 100 000 variables, whereas the RJT framework has $10^6$ constraints and variables, which is 1-2 orders of magnitude more than the RDP formulation. The DP formulation has approximately 2000 constraints and $10^8$ variables, which is 3 orders of magnitude more than in the RDP formulation. Unsurprisingly, the large number of variables combined with the small number of constraints in the DP formulation leads to a poor root node LP relaxation, which is 510 \% larger than with the RJT formulation in instances with $K = 2$ states and the default Gurobi parameters and 501 \% larger than with the RJT formulation and the modified Gurobi parameters. The RJT formulation has the best LP relaxation, although the RDP formulation is not far off with 10-41 \% looser root relaxations with default Gurobi parameters and 66-93 \% looser root relaxations with the modified Gurobi parameters. This observation supports the differences in total time, as the RDP formulation finds an efficient compromise between model size and root node relaxation quality. 

This is reflected in the computation times. With the default Gurobi parameters, the RJT formulation needs approximately 15x more total time than the RDP formulation to solve instances with $K = 3$ chance node states. Moreover, the RJT formulation fails to converge to the optimal solution within the time limit in most instances with $K = 4$ chance nodes states and in all instances with $K = 5$ chance node states and beyond. On the other hand, 55x more total time is needed to find the optimal strategy with the DP formulation than with the RDP formulation in instances with $K = 2$ chance node states. Moreover, the DP formulation fails to converge in all instances with $K = 3$ chance node states and beyond, although a feasible solution for all instances with $K=3$ chance node states is found. With the Gurobi parameter values specified in Table \ref{tab:solver_parameters}, the performance of the RDP framework improves, whereas the performance of the other frameworks suffers. For example, instances with $K = 6$ chance node states are solved 3x faster compared to using the default parameters. With the RJT framework, on the other hand, instances with $K = 3$ chance node states need, on average, 26x more total time than with default Gurobi parameters and 1300x more than with the RDP framework and modified Gurobi parameters. Moreover, within the time limit, an optimal solution is found in only 8 problem instances with $K = 3$ chance node states, and no feasible solutions are found in larger problem instances with the RJT framework. The DP framework is the least efficient formulation as the solver fails to find optimal solutions in some of the smallest analyzed instances with $K = 2$ chance node states and finds no optimal solutions in instances with $K = 3$ chance node states and beyond. This computational experiment showcases that the RDP formulation helps solve instances that are intractable with the earlier formulations.

\begin{table}
\caption
{Results of the turbine inspection and maintenance problem with Gurobi parameters presented in Table \ref{tab:solver_parameters}.} \label{tab:turbine}
\begin{tabular}{@{}l|@{}c@{}@{}c@{}@{}c@{}@{}c@{}@{}c@{}@{}c@{}@{}|c@{}@{}c@{}@{}c@{}@{}c@{}@{}c@{}@{}c@{}@{}|@{}c@{}@{}c@{}@{}c@{}@{}c@{}@{}c@{}@{}c@{}@{}}
\toprule
 & \multicolumn{6}{c}{RDP } & \multicolumn{6}{c}{RJT}   & \multicolumn{6}{c}{DP} \\ \cline{2-7} \cline{8-13} \cline{14-19}
$K$              & $\overline{\tau}_t$(s)  &  $\overline{\tau}_s$(s)  & \#O~ &RR \%~ & \#c~ & \#v  & $\overline{\tau}_t$(s)  & $\overline{\tau}_s$(s) & ~\#O/F~ &RR \%~ &\#c & \#v  & $\overline{\tau}_t$(s)  &  $\overline{\tau}_s$(s)  & ~\#O/F~ &RR \% &\#c & \#v  \\ \midrule 
2 & ~0.05 & 0.02  & 50& +66.5 & 96 & 336~ & 1.34 & 1.25 & 50/50& +0& 1864~& ~1320& ~1526 & ~1526 & 37/50& +501& 64& ~4656 \\ 
3 & ~0.84 & 0.40  & 50& +83.1 & 459~ & ~2349~ & 1104~ & ~1097 & 8/50& +0 & 2e4 & 1e4 & T & T & 0/50& +619 & ~216 & 2e5\\ 
4 & ~11.35 & 5.93  & 50& +92.8 &~1536~ & ~9600~ & OoM & OoM & 0/50 & +0& 1e5 & 9e4 & T & T & 0/50 & +672 & 512& 2e6 \\ 
5 & ~61.86 & ~21.74  & 50 & +0 & 4125~ & ~3e4~ & OoM & OoM & 0/0 & N/A & 4e5 & 4e5 & T & T & 0/0 & N/A& 1000& 2e7  \\ 
6 & ~322.2 & ~115.4  & 50& +0 & 9504& ~7e4~ & OoM & OoM & 0/0& N/A& 1e6& 1e6 & T & T & 0/0&N/A& 1728 & 9e7 \\\bottomrule   

\end{tabular}
\end{table}

\begin{table}
\caption
{Results of the turbine inspection and maintenance problem with default Gurobi parameters.} \label{tab:turbine_gurobi}
\begin{tabular}{@{}l|@{}c@{}@{}c@{}@{}c@{}@{}c@{}@{}|c@{}@{}c@{}@{}c@{}@{}c@{}@{}|@{}c@{}@{}c@{}@{}c@{}@{}c@{}@{}}
\toprule
 & \multicolumn{4}{c}{RDP } & \multicolumn{4}{c}{RJT}   & \multicolumn{4}{c}{DP} \\ \cline{2-5} \cline{6-9} \cline{10-13}
$K$              & $\overline{\tau}_t$ (s)  &  $\overline{\tau}_s$ (s)  & \#O~ &RR \%~  &  $\overline{\tau}_t$ (s)  & $\overline{\tau}_s$ (s) & ~\#O/F~ &RR \%~ & $\overline{\tau}_t$ (s)  &  $\overline{\tau}_s$ (s)  & ~\#O/F~ &RR \%   \\ \midrule 
2 & ~0.11 & ~0.08  & 50 & +10.2 &0.38 & ~0.17 & 50/50 & +0 &  ~6.07 & ~5.89~ & 50/50~ & ~+510  \\ 
3 & ~2.64 & ~2.18  & 50 & +17.9  & 42.21 & ~34.99~ & 50/50 & +0   & ~OoM & OoM & 0/50 & +614  \\ 
4 & ~49.48 & ~43.85  & 50 & +40.2  & 757.8 & ~513.9~ & 10/50 & +0  & ~OoM & OoM & 0/0 & N/A  \\ 
5 & ~302.9 & ~260.5  & 50 & +0 &  OoM~ & ~OoM & 0/0 & N/A &  ~OoM & OoM & 0/0 & N/A  \\ 
6 & ~1111 & ~889.1  & 50 & +0  & OoM & OoM & 0/0 & N/A  & ~~OoM~ & OoM & 0/0 & N/A   \\\bottomrule   

\end{tabular}
\end{table}

The Supplementary material presents another version of the continuous turbine inspection and maintenance problem which demonstrates how a specific instance of the problem with known conditional probability distributions can be discretized and solved. Moreover, it also shows how much the decision strategy improves for larger state spaces enabled by the RDP formulation, compared to the largest possible state space enabled by the RJT formulation and the DP formulation, further showcasing the more efficient formulation.

\subsection{Extended oil wildcatter problem}
\label{sec:oil}

In the following variant of the oil wildcatter problem \citep{Raiffa1968}, a wildcatter decides whether to run a series of geological tests for a potential oil well before deciding whether or not to start drilling. In the simple version of the problem, the decision maker has only one test to choose from. More realistically, there are $M$ tests available. The RDP formulation is assessed by solving extended oil wildcatter problems with a varying number of available tests and with fixed state spaces. In contrast, the previous computational experiment in Section \ref{sec:turbine} kept the number of nodes fixed and varied the number of states. This experiment provides insight into how the RDP formulation compares computationally against the RJT and DP frameworks in solving influence diagrams with multiple nodes having fewer states (i.e., smaller state spaces in each node). The influence diagram of the extended oil wildcatter problem with two tests is illustrated in Figure \ref{fig:Oil} with its nodes described in Table \ref{tab:nodes_oil}.

\begin{figure}[pos=htbp]
\centering 
\begin{tikzpicture}
    [decision/.style={color = black, fill=white!80, draw, minimum size=2.5em, inner sep=2pt}, 
    chance/.style={circle, color = black, fill=white!80, draw, minimum size=2.5em, inner sep=2pt},
    value/.style={diamond, color = black, fill=white!60, draw, minimum size=2.5em, inner sep=2pt},
    scale=1.5]
    %\draw[step=1cm,gray,very thin] (0,0) grid (4,3);
    \node[chance]   (a) at (-2, 2)  {$O$};
    \node[chance]   (s) at (-2, 1)  {$S$};
    \node[chance]   (t1) at (0, 1)  {$R_2$};
    \node[chance]   (t2) at (-4, 1)  {$R_1$};
    \node[decision]   (i1) at (0, 0)  {$T_2$};
    \node[decision]   (i2) at (-4, 0)  {$T_1$};
    \node[decision]   (d) at (-2, 4)  {$D$};
    \node[value]   (v) at (-2, 3)  {$U$};
    \node[value]   (v1) at (0, -1)  {$C_2$};
    \node[value]   (v2) at (-4, -1)  {$C_1$};
    \draw[->, thick] (a) -- (s);
    \draw[->, thick] (s) -- (t1);
    \draw[->, thick] (s) -- (t2);
    \draw[->, thick] (i1) -- (t1);
    \draw[->, thick] (i2) -- (t2);
    \draw[->, thick] (t2) -- (d);
    \draw[->, thick] (t1) -- (d);
    \draw[->, thick] (d) -- (v);
    \draw[->, thick] (a) -- (v);
    \draw[->, thick] (i1) -- (v1);
    \draw[->, thick] (i2) -- (v2);
\end{tikzpicture}
%\vspace{-0.5cm}
\caption{Influence diagram of the extended oil wildcatter problem, with $M=2$ tests.} \label{fig:Oil}
\vspace{-0.5cm}
\end{figure}
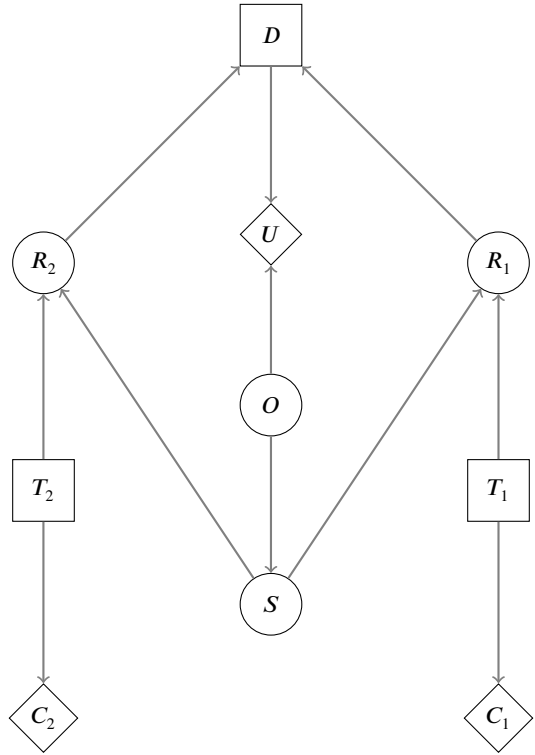

\begin{table}
\caption{Nodes of the extended oil wildcatter problem.}\label{tab:nodes_oil}
\begin{tabular}{@{}l|@{\quad}c|@{}@{\quad}c|@{}@{\quad}c}
\toprule
Node              & Type &  States  & Explanation    \\ \midrule
$O$    & Chance node   & dry, medium, wet & The state of the oil well          \\ 
$S$    & Chance node & dry, medium, wet & Geological features of the drilling site \\ 
$T_i$    & Decision node & yes, no & Decision to commission test $i$ \\ 
$R_i$    & Chance node & N/A, dry, medium, wet & Results of test $i$ \\ 
$C_i$    & Value node & - & Cost of test $i$ \\ 
$D$    & Decision node & yes, no & Drilling decision\\ 
$U$    & Value node & - & Reward minus drilling costs \\ \bottomrule 
%9    &  - & - & 0/50  &  39.971& 7.824  & 50/50 & 130.902 & 127.062 & 50/50 \\ 
%10    & -  & - & 0/50  &  82.33& 10.344  & 50/50 & 63.838  & 55.579 & ~50/50 \down\\ \hline

\end{tabular}
\end{table}

Instances of the extended oil wildcatter problem are created by randomly generating conditional probabilities and utilities. Conditional probabilities are generated as in the turbine inspection and maintenance problem by using the normalized weight vector \eqref{eq:probability_generation_formula}. A moderate global sharpness factor $r = 1$ is used. The most probable state for $S$ is equivalent to the state of node $O$, and the most probable states of the report nodes $R_i$ are equivalent to the state of node $S$. For nodes $R_i$ the probability of the state $N/A$ is $\mathbb{P}(s_{R_i} = N/A \mid s_{T_i} = yes) = 0$ if the test is commissioned or $\mathbb{P}(s_{R_i} = N/A \mid s_{T_i} = no) = 1$ if the test is not commissioned. The testing costs are randomly selected from Uni($1,50$) and the drilling costs from Uni$(50,90)$. The rewards are generated by selecting two random numbers from Uni$(100,1000)$ and ordering them so that $s_O = medium$ is associated with the smaller reward and $s_O = wet$ is associated with the larger reward. The state $s_O = dry$ is associated with zero reward. Moreover, if $s_D = no$, drilling is not started and the reward will be zero. The utility at $U$ is calculated by subtracting the testing and drilling costs from the rewards.

\begin{table}
\caption
{Results of the extended oil wildcatter problem with 50 analyzed instances per problem type using the Gurobi parameters in Table \ref{tab:solver_parameters}.} \label{tab:oil_gurobi}
\begin{tabular}{@{}l|@{}c@{}@{}c@{}@{}c@{}@{}c@{}@{}c@{}@{}c@{}@{}|c@{}@{}c@{}@{}c@{}@{}c@{}@{}c@{}@{}c@{}@{}|@{}c@{}@{}c@{}@{}c@{}@{}c@{}@{}c@{}@{}c@{}@{}}
\toprule
 & \multicolumn{6}{c}{RDP } & \multicolumn{6}{c}{RJT}   & \multicolumn{6}{c}{DP} \\ \cline{2-7} \cline{8-13} \cline{14-19}
$M$              & $\overline{\tau}_t$ (s)  &  $\overline{\tau}_s$ (s)  & \#O/F~ &RR \%~  & \# c & \# v & $\overline{\tau}_t$ (s)  & $\overline{\tau}_s$ (s) & ~\#O/F~ &RR \%~ & \# c & \# v & $\overline{\tau}_t$ (s)  &  $\overline{\tau}_s$ (s)  & ~\#O/F~ &RR \% & \# c & \# v  \\ \midrule 
1 & 0.01 & 0.00  & 50/50 & +83.28 & 19 & 26 & 0.05 & 0.04 & 50/50 & +0 & 196 & 144  & 0.04 & 0.01 & 50/50 & +110.1 & 15 & 154 \\ 
2 & 0.33 & 0.00  & 50/50 & +269.0 & 70 & 164 & 2.73 & 0.14 & 50/50 & +0 & 792 & 604  & 1.23 & 0.08 & 50/50 & +366.3 & 54 & 1188 \\ 
3 & 0.50 & 0.00  & 50/50 & +644.0 & 265 & 1158 & 3.49 & 0.58 & 50/50 & +0 & ~3152~ & ~2432~  & 1.65 & 0.42 & 50/50 & +970.5 & 201 & 9350 \\ 
4 & 0.55 & 0.02  & 50/50 & +1407 & 1036 & 8712 & 8.79 & 3.61 & 50/50 & +0 & 1e4 & 9732  & 12.10 & 9.60 & 50/50 & +2233 & 780 & 7e4 \\ 
5 & 2.18 & 0.27  & 50/50 & +2935 & ~4111~ & ~7e4~ & 65.10~ & ~20.78 & 50/50 & +0 & 5e4 & 4e4  & 156.6 & 136.8 & 50/50 & +4780 & 3087 & 6e5 \\ 
6 & 12.02 & 2.43  & 50/50 & +6014 & 2e4 & 5e5 & 954.4 & 186.9 & 49/50 & +0 & 2e5 & 2e5  & ~2332~ & ~2049~ & 49/50 & +1e4 & 1e4 & 5e6 \\ 
7 & OoM & OoM  & 0/0 & N/A & 7e4 & 4e6 & T & T & 0/0 & N/A & 8e5 & 6e5  & T & T & 0/0 & N/A & 5e4 & 4e7 \\ 
8 & OoM & OoM  & 0/0 & N/A & 3e5 & 3e7 & T& T & 0/0 & N/A & 3e6 & 2e6  & T &T & 0/0 & N/A & 2e5 & 3e8 \\ \bottomrule   

\end{tabular}
\end{table}

\begin{table}
\caption
{Results of the extended oil wildcatter problem with 50 analyzed instances per problem type using the default Gurobi parameters.} \label{tab:oil}
\begin{tabular}{@{}l|@{}c@{}@{}c@{}@{}c@{}@{}c@{}@{}|c@{}@{}c@{}@{}c@{}@{}c@{}@{}|@{}c@{}@{}c@{}@{}c@{}@{}c@{}@{}}
\toprule
 & \multicolumn{4}{c}{RDP } & \multicolumn{4}{c}{RJT}   & \multicolumn{4}{c}{DP} \\ \cline{2-5} \cline{6-9} \cline{10-13}
$M$              & ~$\overline{\tau}_t$ (s)  &  $\overline{\tau}_s$ (s)~  & \#O~ &RR \%~  &  $\overline{\tau}_t$ (s)~  & $\overline{\tau}_s$ (s)~ & ~\#O/F~ &RR \%~ & ~$\overline{\tau}_t$ (s)  &  ~$\overline{\tau}_s$ (s)  & ~\#O/F~ &RR \%   \\ \midrule  
1 & 0.01 & 0.00  & 50/50 & +0  & 0.01 & 0.00 & 50/50 & +2.57   & 0.01 & 0.00 & 50/50 & +118 \\ 
2 & 0.38 & 0.00  & 50/50 & +0 & 2.61 & 0.01 & 50/50 & +2.93   & 1.17 & 0.07 & 50/50 & +389 \\ 
3 & 0.50 & 0.00  & 50/50 & +0  & 2.67 & 0.06 & 50/50 & +2.52  & 1.75 & 0.46 & 50/50 & +980  \\ 
4 & 0.56 & 0.02  & 50/50 & +0 & 5.82 & 0.29 & 50/50 & +1.47  & 14.80 & 12.17 & 50/50 & +2263  \\ 
5 & 2.07 & 0.17  & 50/50 & +0  & 48.37 & 2.05 & 50/50 & +1.03  & 272.4 & 252.5 & 50/50 & +4891~  \\ 
6 & ~10.17 & 0.77  & 50/50 & +0 & 803.8 & 8.29 & 50/50 & +0.96 &  ~1878~ & ~1587 & 50/50 & +1e4  \\ 
7 & ~106.6 & 5.41  & 50/50 & +0  & T & T & 0/0 & N/A   & T & T & 0/0 & N/A  \\ 
8 & OoM & OoM  & 0/0 & N/A  & T & T & 0/0 & N/A   & T & T & 0/0 & N/A  \\ \bottomrule   

\end{tabular}
\end{table}

For each number of tests $M = 1,...,8$, 50 random instances are generated and solved. The results are presented in Tables \ref{tab:oil_gurobi} (obtained with Gurobi parameters in Table \ref{tab:solver_parameters}) and \ref{tab:oil} (obtained with default Gurobi parameters). The RDP formulation again dominates the other formulations in terms of total time, solution time, and solution quality, although the difference is less pronounced compared to the turbine inspection and maintenance problem. The total times are on average 1-2 orders of magnitude smaller than with the RJT formulation or with the DP formulation. 

Additionally, with the RDP formulation, all evaluated instances up to 6-7 tests converge to the optimal solution depending on the used Gurobi parameters, whereas with the RJT formulation and the DP formulation, the largest solved instances have 6 tests. No feasible solutions are found with the RJT framework or the DP framework in instances with 7 or more tests. With the default Gurobi parameters, the RDP formulation reaches its own limit at $M = 8$ tests, where the solver exceeds the available memory; with the modified Gurobi parameters, this limit is reached at $M = 7$ tests. At $M = 8$ tests, the observation variable count is approximately $3\times 10^7$, confirming that the exponential growth in the number of nodes eventually constrains the RDP formulation as well, regardless of the efficiency gains over competing methods.

This time, the observed results cannot be attributed to the compact model size of the RDP framework, which uses more variables in instances of size $M \geq 5$ than the RJT framework. The influence diagram in Figure \ref{fig:Oil} has two chance nodes that do not belong to the observation set, and therefore, the marginal reduction in the number of variables used by the RDP framework versus the DP framework is more moderate compared to the marginal reduction of variables observed in the turbine inspection and maintenance problem. With default Gurobi parameters, the RDP framework starts with the best root node relaxation. However, with the modified Gurobi parameters, the RJT framework clearly finds the best root relaxation value. Still, this is not reflected in the total times or solution times as the RDP framework is largely unaffected by the choice of Gurobi parameters, other than in instances with 7 tests, where the model simply becomes too big without using Gurobi preprocessing and heuristics.

The computational experiments on the extended oil wildcatter problem support the conclusion that the RDP formulation provides computational improvements when solving problems that contain many nodes and few states at chance and decision nodes. Although the solution times improve with the RDP framework, the solution times still grow exponentially as a function of the number of nodes in the influence diagram, which results in computational problems in larger instances.

In the above problem formulation, the decision maker chooses at the outset all the tests that are carried out to inform the drilling decision. This would be the case, for example, when there is limited time for testing before a drilling decision has to be made. If no such limitations exist, tests may be conducted in sequence so that the results of one or more previous tests are available before commissioning further tests. The influence diagram of this sequential variant can be obtained by adding arcs from the results of earlier tests to later testing decisions, in which case the number of required tests and the corresponding cost of testing may be lower. The Supplementary material presents and analyzes an extended oil wildcatter problem with sequential testing decisions. Computationally, however, the number of testing decision variables in the MILP-formulation of the sequential problem grows exponentially, which makes the RJT framework more efficient owing to its superior root node relaxation strength compared to the Decision Programming-based frameworks.

\subsection{Risk-averse water management problem}

As the final computational example, consider a variant of the environmental management problem for the restoration of the Tuusulanjärvi lake reported by \citet{varis1990bayesian}. The problem consists of deciding on a monitoring strategy to track the water quality of the lake and deciding on a dilution strategy based on the monitoring results to improve the water quality of the lake. Compared to the influence diagram in \citet{varis1990bayesian}, this variant contains an additional chance node $A$ that represents exogenous agricultural activity that affects the water quality of the lake. The influence diagram of the problem is in Figure \ref{fig:WaterResource}, and the nodes in Table \ref{tab:nodes_water}. As in the turbine inspection and maintenance problem, the nodes with inherently continuous states are represented by state values within the interval [0, 100]. For example, $s_A = 100$ represents the highest possible agricultural intensity near the inlets of the lake, which produces the largest considered amount of excess fertilizers, which in turn leads to lower water quality. For the water quality before and after dilution, $s_{W_i} = 100$ represents the best possible quality and $s_{W_i} = 0$ represents the worst possible quality. For each problem instance, $K$ equally spaced states from the domain of continuous variables are selected. 

Problem instances are generated, similarly to the two previous instances, by using the weight vector \eqref{eq:probability_generation_formula} for which the noise $\xi_i$ is randomly sampled. Appendix \ref{app:risk-averse} presents the parameters used in creating conditional probabilities. The monitoring costs are generated by sampling three values from Uni$(0,100)$ and associating the largest value with level-3 monitoring and the smallest value with level-1 monitoring. The state $s_C$ represents the dilution cost. Posterior water quality rewards are generated by sampling $K$ values from Uni$(0,500)$, associating the largest reward with the best posterior water quality.

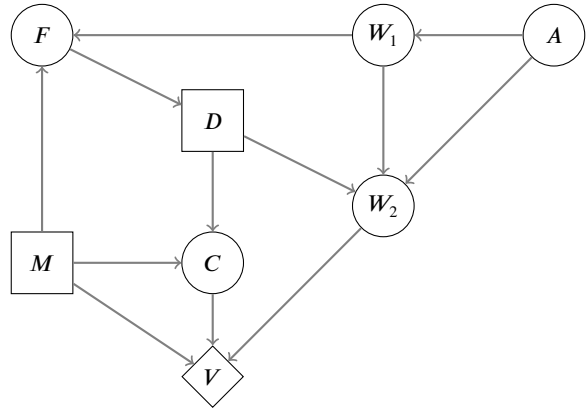
\begin{figure}[pos=htbp]
\centering 
\begin{tikzpicture}
    [decision/.style={color = black, fill=white!80, draw, minimum size=2.5em, inner sep=2pt}, 
    chance/.style={circle, color = black, fill=white!80, draw, minimum size=2.5em, inner sep=2pt},
    value/.style={diamond, color = black, fill=white!60, draw, minimum size=2.5em, inner sep=2pt},
    scale=1.5]
     %\draw[step=1cm,gray,very thin] (0,0) grid (4,3);
     \node[chance] (c) at (1.5, 1)      {$C$};
     \node[chance] (w1) at (3, 3)      {$W_1$};
     \node[chance] (a) at (4.5, 3)      {$A$};
     \node[chance] (w2) at (3, 1.5)      {$W_2$};
     \node[chance] (f) at (0, 3)      {$F$};
     \node[value]  (v) at (1.5, 0)      {$V$};
     \node[decision] (m) at (0, 1)    {$M$};
     \node[decision] (d) at (1.5, 2.25)    {$D$};
     \draw[->, thick] (f) -- (d);
     \draw[->, thick] (w1) -- (w2);
     \draw[->, thick] (d) -- (w2);
     \draw[->, thick] (a) -- (w2);
     \draw[->, thick] (a) -- (w1);
     \draw[->, thick] (w1) -- (f);
     \draw[->, thick] (m) -- (f);
     \draw[->, thick] (d) -- (c);
     \draw[->, thick] (m) -- (c);
     \draw[->, thick] (c) -- (v);
     \draw[->, thick] (m) -- (v);
     \draw[->, thick] (w2) -- (v);
\end{tikzpicture}
%\vspace{-0.5cm}
\caption{ Influence diagram of the risk-averse water resource management problem.} \label{fig:WaterResource}
\vspace{-0.5cm}
\end{figure}

\begin{table}
\caption
{Nodes of the risk-averse water management problem. }\label{tab:nodes_water}
\begin{tabular}{@{}l|@{\quad}c|@{}@{\quad}c|@{}@{\quad}c}
\toprule 
Node              & Type &  States  & Explanation    \\ \midrule 
$A$    & Chance node   & [0,100] &   Agricultural intensity        \\ 
$W_1$    & Chance node & [0,100] & Water quality before dilution \\ 
$W_2$    & Chance node & [0,100] & Water quality after dilution \\ 
$F$    & Chance node & [0,100] & Water quality monitoring results \\ 
$M$    & Decision node & level 1, level 2, level 3 & Water quality monitoring strategy \\ 
$D$    & Decision node & [0,100] & Amount of clean water admitted into the lake\\ 
$C$    & Chance node & [0,100] & Cost of dilution\\ 
$V$    & Value node & - & Rewards minus monitoring and dilution costs\\ \bottomrule  
%9    &  - & - & 0/50  &  39.971& 7.824  & 50/50 & 130.902 & 127.062 & 50/50 \\ 
%10    & -  & - & 0/50  &  82.33& 10.344  & 50/50 & 63.838  & 55.579 & ~50/50 \down\\ \hline

\end{tabular}{}
\end{table}

In contrast to the examples in Sections \ref{sec:turbine} and \ref{sec:oil} where influence diagrams were solved by maximizing the expected utility, the RDP-CVaR formulation is used to find the strategy that minimizes risks by maximizing the CVaR given a probability threshold $\alpha = 0.2$, which corresponds to a confidence level of 80\%. Solution times are compared by solving the same instances with the CVaR formulation of the RJT framework presented in \citet{herrala2025risk} and with the CVaR formulation for Decision Programming, which uses the optimization model presented in \citet{hankimaa2025solving} with constraints \eqref{eq:cvar-1}-\eqref{eq:cvar-11} and the objective function defined as in \eqref{eq:cvar}. This model is referred to as DP-CVaR.

\begin{table}
\caption
{Results of the risk-averse water resource management problem with 50 analyzed instances per problem type using the Gurobi parameters in Table \ref{tab:solver_parameters}.} \label{tab:water_cvar}
\begin{tabular}{@{}l|@{}c@{}@{}c@{}@{}c@{}@{}c@{}@{}c@{}@{}c@{}@{}|c@{}@{}c@{}@{}c@{}@{}c@{}@{}c@{}@{}c@{}@{}|@{}c@{}@{}c@{}@{}c@{}@{}c@{}@{}c@{}@{}c@{}@{}}
\toprule
 & \multicolumn{6}{c}{RDP-CVaR } & \multicolumn{6}{c}{RJT-CVaR}   & \multicolumn{6}{c}{DP-CVaR} \\ \cline{2-7} \cline{8-13} \cline{14-19}
$K$              & $\overline{\tau}_t$ (s)  &  $\overline{\tau}_s$ (s)  & \#O/F~ &RR \%~  & \# c & \# v & $\overline{\tau}_t$ (s)  & $\overline{\tau}_s$ (s) & ~\#O/F~ &RR \%~ & \# c & \# v & $\overline{\tau}_t$ (s)  &  $\overline{\tau}_s$ (s)  & ~\#O/F~ &RR \% & \# c & \# v  \\ \midrule 
2 & 0.03 & 0.01  & 50 & +0 & 109 & 44 & 0.1 & 0.09 & 50/50 & +0 & 480 & 278  & 1.03 & 0.67 & 50/50 & +0 & 107 & 224 \\ 
3 & 0.1 & 0.09  & 50 & +0.1 & 236 & 94 & 1.95 & 1.43 & 50/50 & +0 & ~2062~ & ~1429~  & 35.39 & 34.64 & 50/50 & +0.8 & 233 & 2254 \\ 
4 & 0.23 & 0.22  & 50 & +0.2 & 413 & 164 & 13.01 & 11.98 & 50/50 & +0 & ~6582~ & ~5032~  & ~2212~ & ~2210~ & 40/50 & +1.0 & 409 & 1e4 \\ 
5 & 0.36 & 0.35  & 50 & +0.3 & 640 & 254 & 43.28 & 36.28 & 50/50 & +0 & 2e4 & 1e4  & ~T~ & ~T & 0/8 & +2.5 & 635 & 5e4 \\ 
6 & 0.59 & 0.57  & 50 & +0.3 & 917 & 364 & 158.4 & 124.0 & 50/50 & +0 & 4e4 & 3e4  & T & T & 0/0 & +2.7 & 911 & 1e5 \\ 
7 & 0.87 & 0.83  & 50 & +0.3 & 1244 & 494 & ~534.1~ & 407.2 & 50/50 & +0 & 8e4 & 7e4  & T & T & 0/0 & +2.7 & 1237 & 4e5 \\ 
8 & 1.26 & 1.07  & 50 & +0.3 & 1621 & 644 & ~1960~ & ~1452~ & 41/49 & +0 & 1e5 & 1e5  & T & T & 0/0 & N/A & 1613 & 8e5 \\ 
9 & 1.51 & 1.29  & 50 & +0.3 & 2048 & 814 & T & T & 0/1 & +0 & 2e5 & 2e5  & T & T & 0/0 & N/A & 2039 & 2e6 \\ 
10 & 2.20 & 1.75  & 50 & +0 & 2525~ & 1004~ & T & T & 0/0 & N/A & 4e5 & 4e5  & T & T & 0/0 & N/A & 2515 & 3e6 \\ \bottomrule   

\end{tabular}
\end{table}

\begin{table}
\caption
{Results of the risk-averse water resource management problem with 50 analyzed instances per problem type using default Gurobi parameters.} \label{tab:water_cvar_gurobi}
\begin{tabular}{@{}l|@{}c@{}@{}c@{}@{}c@{}@{}c@{}@{}|c@{}@{}c@{}@{}c@{}@{}c@{}@{}|@{}c@{}@{}c@{}@{}c@{}@{}c@{}@{}}
\toprule
 & \multicolumn{4}{c}{RDP-CVaR} & \multicolumn{4}{c}{RJT-CVaR}   & \multicolumn{4}{c}{DP-CVaR} \\ \cline{2-5} \cline{6-9} \cline{10-13}
$K$              & $\overline{\tau}_t$ (s)  &  $\overline{\tau}_s$ (s)  & \#O~ &RR \%~  &  $\overline{\tau}_t$ (s)  & $\overline{\tau}_s$ (s) & ~\#O/F~ &RR \%~ & $\overline{\tau}_t$ (s)  &  $\overline{\tau}_s$ (s)  & ~\#O/F~ &RR \%   \\ \midrule  
2 & 0.01 & 0.00  & 50 & +7.13  & 0.02 & 0.01 & 50/50 & +0   & 0.13 & 0.09 & 50/50 & +1392  \\ 
3 & 0.02 & 0.01  & 50 & +0  & 0.66 & 0.07 & 50/50 & +128.7   & 14.26 & 13.98 & 50/50 & +1080 \\ 
4 & 0.05 & 0.04  & 50 & +0  & 1.30 & 0.25 & 50/50 & +65.46  & ~392.4~ & ~389.8 & 50/50 & +78.89 \\ 
5 & 0.10 & 0.09  & 50 & +0  & 8.10 & 1.06 & 50/50 & +4.39   & ~3079~ & ~3066 & 1/39 & +8.34 \\ 
6 & 0.23 & 0.21  & 50 & +0.05  & 38.36 & 3.58 & 50/50 & +0.84   & T & T & 0/36 & +3.73  \\ 
7 & 0.39 & 0.34  & 50 & +0.06  & 150.2~ & 13.93 & 50/50 & +0.34   & T & T & 0/22 & +2.70 \\ 
8 & 0.71 & 0.60  & 50 & +0.18  & ~587.6~ & 51.20 & 50/50 & +0.07   & T & T & 0/0 & N/A  \\ 
9 & 1.09 & 0.87  & 50 & +0.16  & ~3080~ & ~207.6 & 37/38 & +0.05   & T & T & 0/0 & N/A \\ 
10 & 1.71 & 1.23  & 50 & +0  & T & T & 0/0 & N/A   & T & T & 0/0 & N/A  \\ \bottomrule   

\end{tabular}
\end{table}

Table \ref{tab:water_cvar} reports The results of 50 randomly generated instances with $K \in 2,\ldots,10$ states are reported in Tables \ref{tab:water_cvar} (obtained with the Gurobi parameters presented in Table \ref{tab:solver_parameters}) and \ref{tab:water_cvar_gurobi} (obtained with the default Gurobi parameters). They show that the RDP-CVaR formulation significantly improves the computational time needed for finding the CVaR maximizing strategy regardless of the used Gurobi parameters with the performance gap widening sharply as $K$ grows. The RDP-CVaR formulation solves all analyzed instances in seconds, regardless of the used Gurobi parameters. In contrast, the RJT-CVaR formulation and DP-CVaR formulation become intractable when the state spaces of chance and decision nodes become large enough. With default parameters, RJT-CVaR formulation reaches optimality in only 37 of 50 instances with $K=9$ states (averaging 3080s) and none in instances with $K=10$ states; DP-CVaR reaches optimality for just 1 of 50 instances with $K=5$ states (3079s) and none beyond. RDP-CVaR, by contrast, requires on average 1.09s and 0.10s to solve the corresponding instance sizes. With the Gurobi parameters presented in Table \ref{tab:solver_parameters}, the RJT-CVaR formulation solves 41 instances with $K = 8$ states to optimality (averaging 1960 seconds), with no instances larger than this solved to optimality within the time limit. The DP-CVaR formulation solves 40 instances with $K = 4$ states to optimality and none beyond. The RDP formulation, in contrast, requires on average 1.26 and 0.23 seconds to find the optimal decision strategy in these instances.

In the analyzed problem instances, the RDP-CVaR formulation finds good root relaxations, except in instances with $K = 2$ states with the default Gurobi parameters. In all other instances, the root relaxation values are on average within 0.3\% of the best found root relaxation. Interestingly, no model consistently finds the best root node relaxations, as in some instances the RDP-CVaR model finds the tightest relaxation, and in other instances, the RJT-CVaR model finds the tightest relaxation.

The main driver of the observed total times is again the compact model size of the RDP-CVaR formulation. The number of constraints used by the RDP-CVaR model ranges from 109 in instances with $K = 2$ states to 2525 in instances with $K = 10$ states. Similarly, the number of variables in the corresponding models ranges from 44 to 1004, suggesting that the size of the optimization model grows moderately with respect to $K$. In contrast, the number of constraints used by the RJT-CVaR formulation ranges from 480 to 400 000, which is 4.5-160x more than the corresponding number of constraints required by the RDP-CVaR formulation. Similarly, the number of used variables ranges from 278 to 400 000, which is 6.5-400x more than used by the RDP-CVaR formulation. The number of constraints used by the DP-CVaR formulation is almost equivalent to the number of constraints used by the RDP-CVaR formulation, but the number of variables is considerably larger, with 5-3000x more than what is used in the RDP-CVaR formulation.  

Taken together, these computational experiments show that the RDP-CVaR formulation significantly improves the efficiency in finding the CVaR-maximizing decision strategy. Instances that have been intractable for the earlier formulations can be solved to optimality in a few seconds using the RDP-CVaR formulation. In addition, the total time for the RDP-CVaR formulation does not grow exponentially, suggesting that considerably larger instances of the risk-averse water resource management problem can be solved to optimality within a reasonable time with the RDP formulation. 

\section{Conclusions}\label{sec:Conclusions}

In this paper, we have presented a novel MILP formulation that is computationally much more efficient than the original Decision Programming formulation in \citep{salo2022decision}. This significant improvement in efficiency is gained by aggregating path variables and introducing a valid inequality. The wide applicability of MILP models is retained by admitting a wide range of objectives and constraints. For example, the CVaR risk measure can be employed in the objective function or chance constraints. 

To illustrate the computational efficiency of the proposed RDP formulation,  we have solved several variants of problems from the literature and also compared its solution times to those obtained by employing state-of-the-art formulations discussed in  \citet{hankimaa2025solving,herrala2025risk,parmentier2020integer}. In particular, the RDP formulation usually outperforms these earlier formulations in solving influence diagrams that cannot be cast as POMDPs and that have nodes with large state spaces. Among other things, this makes it possible to provide better solutions for influence diagrams that contain continuous variables, as these variables can be discretized more accurately. In this sense, the RDP formulation can be used to solve influence diagrams that have been previously intractable.

These improvements notwithstanding, the size of the MILP models based on the RDP formulation grows exponentially with respect to the total number of decision nodes and chance nodes whose values are known when making decisions. As this limits the size of the influence diagrams that can be solved, continued research is needed to develop innovative MILP-based approaches to solve the resulting optimization models more efficiently, for example, by introducing valid inequalities or by developing specialized solution algorithms, such as decomposition approaches. In addition, other ways of aggregating variables in building the MILP model can be pursued to obtain further efficiency gains. 

Finally, our computational experiments show that no formulation consistently outperforms the others across all problem types. Consequently, there is a need for more specific guidance on which formulation is most likely to be the most efficient for the problem at hand. Such guidance can be offered by developing analytical evaluation frameworks and further computational experiments to make well-founded choices among alternative MILP formulations. 

%The RDP formulation fully elaborates the tentative ideas largely was originally inspired by the discussion in  \citep{salo2022decision}, who pointed to the possibility of aggregating the information needed for defining optimal decisions as a direction for development. 
%\input{Manuscript/sections/6-Code_data_disclosure}

% Use the style of numbering in square brackets.
% If nothing is used, default style will be taken.
%\begin{enumerate}[a)]
%\item 
%\item 
%\item 
%\end{enumerate}  

%\begin{itemize}
%\item 
%\item 
%\item 
%\end{itemize}  

%\begin{description}
%\item[]
%\item[] 
%\item[] 
%\end{description}  

% Uncomment and use as the case may be
%\begin{theorem} 
%\end{theorem}

% Uncomment and use as the case may be
%\begin{lemma} 
%\end{lemma}

 \appendix
 \section{Proofs}
\label{sec:proofs}

\noindent\textbf{Proof of Proposition \ref{prop:suboptimal_solution}}

First, it is shown that $(x^{*},z)$ fulfills constraints \eqref{eq:dp_strategy1}-\eqref{eq:dp_path_compat}. Because $(x,z)$ is assumed to fulfill the constraints \eqref{eq:dp_strategy1}-\eqref{eq:dp_path_compat}, Constraints \eqref{eq:dp_strategy1} are trivially fulfilled by $(x^{*},z)$. Similarly, Constraints \eqref{eq:dp_strategy} and \eqref{eq:dp_path_compat}  are trivially fulfilled. Then, it remains to show that $(x^{*},z)$ fulfills Constraint \eqref{eq:dp_path_compat_strategy}. 

Consider $s \in S^>$. If $x^{*}(s) = 0$, Constraints \eqref{eq:dp_path_compat_strategy} are trivially fulfilled. If $x^{*}(s) = 1$, then there exists $s' \in E(s_O)$ such that $x(s') > 0$. From the assumption that $(x,z)$ is feasible for \eqref{eq:dp_strategy1}-\eqref{eq:dp_path_compat}, it necessarily follows that $z(s'_d \mid s'_{I(d)}) = 1, \forall d \in D$, for this $s'$. Also note that $d \in O$ and $I(d) \subsetneq O, \forall d \in D$. Moreover, because $s' \in E(s_O)$, it also holds that $s'_O = s_O$. Therefore, it follows that $z(s'_d \mid s'_{I(d)}) = z(s_d \mid s_{I(d)}) = 1, \forall d \in D$. Hence, Constraint \eqref{eq:dp_path_compat_strategy} is also fulfilled, and $(x^{*},z)$ is feasible for \eqref{eq:dp_strategy1}-\eqref{eq:dp_path_compat}.

Next, the suboptimality of $(x,z)$ is proved. Recall that $U(s) > 0, \forall s \in S^{>}$ and $p(s) > 0 , \forall s \in S^{>}$. From assuming $x^{*}(s) = \max_{s' \in E^{>}(s_O)}x(s')$ it follows that $x^{*}(s) \geq x(s), \forall s \in S^{>}$. Therefore, it follows that $p(s)U(s)x^{*}(s) \geq p(s)U(s)x(s), \forall s \in S^{>}$. Moreover, by assumption that $\exists s \in S^{>} \text{ s.t. } x(s) \neq x^{*}(s)$ there exists at least one $s \in S^{>}$ such that $x^{*}(s) > x(s)$. Then, $\sum_{s \in S^{>}}p(s)U(s)x^{*}(s) > \sum_{s \in S^{>}}p(s)U(s)x(s)$ follows.

\vspace{12pt}

\noindent\textbf{Proof of Proposition \ref{prop:objectives_match}}

Notice the following:
\begin{align*}
&\sum_{s \in S_O}\mathbb{E}_U(s_O)y(s_O) = 
\sum_{s \in S_O}\biggr[\sum_{s' \in E(s_O)}p(s')U(s')\biggr]y(s_O) = \sum_{s \in S_O}\sum_{s' \in E(s_O)}p(s')U(s')y(s_O) =\\ 
&\sum_{s \in S_O}\biggr[\sum_{s' \in E(s_O) \setminus S^{>}}p(s')U(s')y(s_O) + \sum_{s' \in E^{>}(s_O)}p(s')U(s')y(s_O)\biggr]=\\ 
&\sum_{s \in S_O}\sum_{s' \in E^{>}(s_O)}p(s')U(s')y(s_O) = 
\sum_{s \in S_O}\biggr[\sum_{s' \in E^{>}(s_O)}p(s')U(s')x(s')\biggr] =\\ 
&\sum_{s \in S^{>}}p(s)U(s)x(s).
\end{align*}
The first equality follows directly from the definition of $\mathbb{E}_U(s_O)$. The second equality simply moves $y(s_O)$ inside the inner sum. The third equality splits the inner sum. The fourth equality follows from the definition of $S^{>}$ so that $p(s') = 0, \forall s' \in E(s_O) \setminus S^{>} $. The fifth equality follows based on \eqref{eq:reduced}, which reduces to $x(s') = y(s_O), \forall s' \in E(s_O)$. The last equality simply combines the sums.

\vspace{12pt}

\noindent\textbf{Proof of Proposition \ref{prop:equivalence}}

Assume that $(y, z)$ is a feasible solution for \eqref{eq:dpr_strategy1}-\eqref{eq:dpr_path_compat}. Constraints \eqref{eq:dpr_strategy1} and \eqref{eq:dp_strategy1} are equivalent, and therefore, by assumption $z(s_d \mid s_{I(d)})$ fulfills \eqref{eq:dp_strategy1}. Additionally, Constraints \eqref{eq:dp_strategy} and \eqref{eq:dp_path_compat} hold trivially.

Let $d \in D$, $s_d \in S_d$, and $s_{I(d)} \in S_{I(d)}$. For Constraints \eqref{eq:dp_path_compat_strategy}, assume that $y(s_{O}) = 1$ for some $s_{O} \in E_{O}(s_{\{d\} \cup I(d)})$. Then, according to constraints \eqref{eq:dpr_path_compat_strategy}, $z(s_d \mid s_{I(d)}) = 1$. Then, the corresponding Constraints \eqref{eq:dp_path_compat_strategy} for $d, s_d$, and $s_{I(d)}$ reduce to 
\begin{align*}
\sum_{s \in E(s_{\{d\} \cup I(d)})}x(s) \leq \Gamma(s_d \mid s_{I(d)}).
\end{align*}

According to \eqref{eq:reduced}, $x(s) = 1, \forall s \in E(s_{O})$. From the assumption $s_{O} \in E_{O}(s_{\{d\} \cup I(d)})$ it follows that $s \in E(s_{\{d\} \cup I(d)}), \forall s \in E(s_{O})$. As $\Gamma(s_d \mid s_{I(d)})$ is assumed to be a bigM-constant for \eqref{eq:dp_objective}-\eqref{eq:dp_path_compat}, $x(s) = 1, \forall s \in E(s_{O})$ is a feasible solution. Conversely, assume that $y(s_{O}) = 0$ for some $s_{O}$. Then, $x(s) = 0, \forall s \in E(s_{O})$. Then, Constraints \eqref{eq:dp_path_compat_strategy} are trivially fulfilled for any feasible $z(s_d \mid s_{I(d)})$.

\vspace{12pt}

\noindent\textbf{Proof of Proposition \ref{prop:optimality}}

Assume that $(y^{*}, z^{*})$ is an optimal solution for \eqref{eq:dpr_objective}-\eqref{eq:dpr_path_compat}. First, due to Proposition \ref{prop:equivalence}, $(x^{*},z^{*})$ is a feasible solution that satisfies \eqref{eq:dp_strategy1}-\eqref{eq:dp_path_compat}. By an optimality argument, for any $(y, z)$, which is feasible for \eqref{eq:dpr_objective}-\eqref{eq:dpr_path_compat}, it follows that
$$
\sum_{s_{O} \in S_{O}}\mathbb{E}_U(s_{O})y(s_{O}) \leq \sum_{s_{O} \in S_{O}}\mathbb{E}_U(s_{O})y^{*}(s_{O}).
$$

Let $(x, z)$ satisfy \eqref{eq:dp_strategy1}-\eqref{eq:dp_path_compat}. Let $x'(s) = \max_{s' \in E(s_O)}x(s')$. Then, the following holds true:
\begin{align*}
    \sum_{s \in S^{>}}p(s)U(s)x(s) &\leq 
    \sum_{s \in S^{>}}p(s)U(s)x'(s) =
    \sum_{s \in S}p(s)U(s)x'(s) =
    \sum_{s_{O} \in S_{O}}\mathbb{E}_U(s_O)y'(s_{O}) \leq\\ \sum_{s_{O} \in S_{O}}\mathbb{E}_U(s_{O})y^{*}(s_{O}) &=
    \sum_{s \in S^{>}}p(s)U(s)x^{*}(s).
\end{align*}

Notice that if $x(s) = 1$ for any $s \in S^{>}$, then $x'(s') = 1, \forall s' \in E(s_O)$ and therefore, $x'$ can be represented with $y'$ as defined in \eqref{eq:reduced} based on $x'$ without loss of generality. The first inequality follows as $p(s)U(s) \geq 0, \forall s \in S$ and $x'(s) \geq x(s), \forall s \in S$. The first equality holds due to adding terms for paths $s \in S \setminus S^{>}$, for which $p(s) = 0$. The second equality follows from Proposition \ref{prop:objectives_match}. The second inequality follows from the assumption of optimality of $y^{*}$ and the last equality follows by applying Proposition \ref{prop:objectives_match} and transformation \eqref{eq:reduced}. Therefore, $(x^{*},z^{*})$ is optimal for \eqref{eq:dp_objective}-\eqref{eq:dp_path_compat}.

\vspace{12pt}

\noindent\textbf{Proof of Proposition \ref{prop:valid_inequality}}

Consider $\bar{s}_{C_I} \in S_{C_I}$ and $s_O, s'_O \in E_O(\bar{s}_{C_I})$ so that $s_O \neq s'_O$. Then, $\exists d \in D$ s.t. $s_d \neq s'_d$. Assume that $y(s_{O}) = 1$. The assumption $y(s_O) = 1$ implies that $z(s_d \mid s_{I(d)}) = 1$ for all $d \in D$, where  $s_d$ and $s_{I(d)}$ are extracted from $s_O$. To prove the claim, it is shown that there exists $d \in D$ such that $z(s'_d \mid s'_{I(d)}) = 0$, where $s'_d$ and $s'_{I(d)}$ are extracted from $s'_O$. If this is the case, Constraints \eqref{eq:dpr_path_compat_strategy} ensure that $y(s'_O) = 1$ is not a feasible solution for constraint set \eqref{eq:dpr_objective}-\eqref{eq:dpr_path_compat}.

Denote $I_C(d) = C \cap I(d)$ and $I_D(d) = D \cap I(d)$ as the chance nodes and decision nodes contained in an information set of $d \in D$. By assumption, value nodes cannot be parents of other nodes, $I(d) = I_C(d) \cup I_D(d)$. If $I_D(d) = \emptyset$, it means that $z(s'_d \mid s'_{I(d)}) = 0$. To see this, note that $I_C(d) \subseteq C_I$ and subsequently $s'_{I_C(d)} = s_{I_C(d)}$. Consequently $s_{I(d)} = s'_{I(d)}$. By assumption of $y(s_O) = 1$, it holds that $z(s_d \mid s_{I(d)}) = 1$. Constraints \eqref{eq:dpr_strategy1} then force $z(s'_d \mid s_{I(d)}) = 0$. Then, according to  \eqref{eq:dpr_path_compat_strategy}, $y(s'_{O})$ would be forced to take value $0$.

If $z(s'_d \mid s'_{I(d)}) = 1$, it then necessarily means that $I_D(d) \neq \emptyset$ and $\exists d' \in I_D(d)$ such that $s'_{d'} \neq s_{d'}$. A similar inspection shows that if  $z(s'_{d'} \mid s'_{I(d')}) = 1$, it then follows that $I_D(d') \neq \emptyset$ and $\exists d'' \in I_D(d')$ such that $s'_{d''} \neq s_{d''}$. As the influence diagram is assumed to be acyclic, this process eventually finds $d^{*}$ such that $s_{d^{*}} \neq s'_{d^*}$ and $I_D(d^{*}) = \emptyset$. Then, by assumption and Constraint \eqref{eq:dpr_strategy1}, $z(s'_{d^{*}} \mid s'_{I(d^{*})}) = 0$ and consequently, $y(s'_{O}) = 0$.

\vspace{12pt}

\noindent\textbf{Proof of Proposition \ref{prop:condition_holds}}

The claim is proved by using a contradiction. Assume that $(y,z)$ is a feasible solution for \eqref{eq:dpr2_objective}-\eqref{eq:dpr2_path_compat} where \eqref{eq:dpr2_valid_ineq} are equality constraints. Denote $I_C(d) = C \cap I(d)$ and $I_D(d) = D \cap I(d)$ as the chance nodes and decision nodes contained in an information set of $d \in D$. Assume that $\exists s_O \in S_O$ such that $y(s_O) = 0$ and $z(s_d \mid s_{I(d)}) = 1, \forall d \in D$. This would contradict \eqref{eq:condition_for_cvar}. Then, based on Constraints \eqref{eq:dpr2_valid_ineq}, $\exists s'_O \in E_O(s_{C_I})$ such that $y(s'_O) = 1$, where $s_{C_I}$ is constructed with states from $s_O$. Next, it is shown that $\exists d \in D$ such that $z(s'_d \mid s'_{I(d)}) = 0$, which would make $y(s'_O) = 1$ infeasible. 

By assumption $s'_O \neq s_O$ and $s'_{C_I} = s_{C_I}$. Then necessarily $\exists d \in D$ such that $s'_d \neq s_d$. If $I_D(d) = \emptyset$, it means that $z(s'_d \mid s'_{I(d)}) = 0$. To see this, note that $s'_{I_C(d)} = s_{I_C(d)}$, and consequently $s_{I(d)} = s'_{I(d)}$. By assumption $z(s_d \mid s_{I(d)}) = 1$. Constraints \eqref{eq:dpr2_strategy1} then force $z(s'_d \mid s_{I(d)}) = 0$. Then, according to constraint \eqref{eq:dpr2_path_compat_strategy}, $y(s'_{O})$ would be forced to take value $0$, which would then be a contradiction.

If $z(s'_d \mid s'_{I(d)}) = 1$, it then necessarily means that $I_D(d) \neq \emptyset$ and $\exists d' \in I_D(d)$ such that $s'_{d'} \neq s_{d'}$. A similar inspection shows that if  $z(s'_{d'} \mid s'_{I(d')}) = 1$, it then follows that $I_D(d') \neq \emptyset$ and $\exists d'' \in I_D(d')$ such that $s'_{d''} \neq s_{d''}$. As the influence diagram is assumed to be acyclic, this process eventually finds $d^{*}$ such that $s_{d^{*}} \neq s'_{d^*}$ and $I_D(d^{*}) = \emptyset$. Then, by assumption and Constraints \eqref{eq:dpr2_strategy1}, $z(s'_{d^{*}} \mid s'_{I(d^{*})}) = 0$ and consequently, $y(s'_{O}) = 0$. This shows that $y(s'_O) = 0, \forall s'_O \in E_O(s_{C_I}) \setminus \{s_O\}$. Then, the only way to fulfill Constraints \eqref{eq:dpr2_valid_ineq} is to have $y(s_O) = 1$, which proves the claim.

\section{Conditional value-at-risk model}
\label{sec:full_cvar_model}

The optimization model \eqref{eq:cvar_ob}-\eqref{eq:cvar_end} leads to the decision strategy that maximizes CVaR at probability level $\alpha$.

%\parbox{80mm}{\small
\begin{subequations}
\begin{align}
    \max_{\substack{z,y, \rho, \overline{\rho},\\ \eta, \lambda, \overline{\lambda}}} &~\frac{1}{\alpha}\sum_{u \in \mathcal{U}}\overline{\rho}(u)u \label{eq:cvar_ob}\\
    \text{s.t.} &\sum_{s_d \in S_{d}}z(s_d \mid s_{I(d)}) = 1,  ~&\forall d \in D, s_{I(d)} \in S_{I(d)} \\ 
    & \sum_{s_{O} \in E_{O}(s_d,s_{I(d)})} y(s_{O}) \leq \Gamma(s_d \mid s_{I(d)}) z(s_d \mid s_{I(d)}),~&\forall d \in D, s_d \in S_d, s_{I(d)} \in S_{I(d)}\\
    & \sum_{s_O \in E_O(s_{C_I})}y(s_O) = 1, ~&\forall s_{C_I} \in S_{C_I}\\
    &\eta - u \le M \lambda(u), \  ~&\forall u \in \mathcal{U} \\ 
        &\eta - u \ge (M+\epsilon) \lambda(u) - M, \  ~&\forall u \in \mathcal{U} \\
    &\eta - u \le (M+\epsilon)\overline{\lambda}(u) - \epsilon, \  ~&\forall u \in \mathcal{U} \\ 
    &\eta - u \ge M (\overline{\lambda}(u)-1), \  ~&\forall u \in \mathcal{U} \\ 
    &\overline{\rho}(u) \le \overline{\lambda}(u), \  ~&\forall u \in \mathcal{U} \\
    & q(u) - (1-\lambda(u)) \le \rho(u) \le \lambda(u) , \  ~&\forall u \in \mathcal{U} \\
    &\rho(u) \le \overline{\rho}(u) \le q(u), \  ~&\forall u \in \mathcal{U} \\ 
    &\sum_{u \in \mathcal{U}} \overline{\rho}(u) = \alpha \\
    &\lambda(u), \overline{\lambda}(u) \in \{0,1\}, \  ~&\forall u \in \mathcal{U} \\ 
    & y(s_{O}) \in [0,1], ~&\forall s_{O} \in S_{O}\\
    & z(s_d \mid s_{I(d)}) \in \{0,1\},~&\forall d \in D, s_d \in S_d, s_{I(d)} \in S_{I(d)} \\
    &\rho(u), \overline{\rho}(u) \in [0,1], \  ~&\forall u \in \mathcal{U} \\
    &\eta \in \mathbb{R} \label{eq:cvar_end}. 
\end{align}
\end{subequations}
%}

\section{Computational experiments additional information}
\label{sec:used_parameters}

\subsection{Probability generation parameters}

\subsubsection{Turbine inspection and maintenance problem}

Table \ref{tab:parameters} presents the used parameter values in the turbine inspection and maintenance problem. The most probable sensor and turbine states are good when the environment conditions are good and fired hours are low and vice versa. The most probable sensor state estimate is equivalent to the true sensor state. The sharpness of the decay in probabilities depends on the true states of the sensor and turbine; the better the states of the sensor and turbine, the larger the sharpness factor $r$. Consequently, the probability distribution is sharper around the true state of the sensor. Similarly, the most probable turbine state estimate is the true turbine state and a bad sensor state estimate flattens the probability distribution, therefore making the estimate more unreliable. If no inspection is selected ($s_{IN} = 1$), then $s_{SR}$ and $s_{TR}$ are the corresponding estimates $s_{SE}$ and $s_{TE}$. This is done with a large sharpness factor, which essentially enforces all other probabilities to 0 other than the states corresponding to the observed sensor and turbine state estimates. If the sensor is inspected, the most probable state is the true state of the system with a sharp probability. If the turbine is inspected, it is assumed that some information can be derived from the sensor condition as well, so the probability distribution for the sensor result is centered around the true sensor state but with a smaller sharpness factor. If the sensor is inspected and deemed of good quality, the turbine state is estimated more accurately based on the sensor reading alone. Therefore, if $s_{IN} = 2$ the probability distribution for $s_{TR}$ is centered around $s_{TS}$ with a sharpness factor that depends on $s_{SR}$. If the turbine is inspected, turbine estimate gives an accurate result with some chance for misclassifications, which is reflected by the sharpness factor 2. Turbine flow is centered around a combination of the true turbine state and the selected maintenance decision. The turbine state sets a baseline for the flow and a maintenance decision then shifts the center of the probability distribution to a more positive direction. Bad exogenous environmental conditions make the turbine flow state more unreliable by decreasing the sharpness factor.

\begin{table}[h!]
\caption{Used parameters in the turbine inspection and maintenance problem.}\label{tab:parameters}
\begin{tabular}{@{}l@{\quad}c@{}@{\quad}c@{}@{\quad}c@{}@{\quad}c@{}}
\toprule
Node              & Information set  & states &   $r$ & $\bar s$       \\ \midrule
$SS$    & $FH$   &    &   1  & $K + 1 -s_{FH}$  \\ 
\hline \\
$TS$    & $FH, W$  &    &   1    &  $K + 1 -\lceil \frac{s_{FH} + s_{W}}{2}\rfloor$ \\
\hline \\
$SE$    & $SS, TS$ &     &  $\frac{2s_{SS}-2}{K-1} \left(0.7 + \frac{0.3(s_{TS}-1)}{K-1}\right) + 0.3$    &  $s_{SS}$   \\
\hline \\
$TE$    & $TS, SE$ &     &   $0.5 + \frac{1.5 s_{SE} - 1.5}{K - 1}$   &   $s_{TS}$  \\
\hline \\
$SR$    & $IN, SS, SE$& $s_{IN} = 1$  & 200   & $s_{SE}$ \\
$SR$    & $IN, SS, SE$& $s_{IN} =2 $  &  2  & $s_{SS}$\\
$SR$    & $IN, SS, SE$& $s_{IN} =3 $  &  1  &$s_{SS}$ \\
\hline \\
$TR$    & $IN, TS, TE, SR$& $s_{IN} = 1$ &  200   & $s_{TE}$   \\
$SR$    & $IN, SS, SE$& $s_{IN} = 2$  &  $0.5 + 0.5\left(\frac{s_{SR}-1}{K-1}\right)$  & $s_{SS}$\\
$SR$    & $IN, SS, SE$& $s_{IN} = 3$  &  $2$  & $s_{SS}$\\
\hline \\
$TF$    & $M, TS, W$      & &  $2\left(1 - \frac{0.6(s_W-1)}{K-1}\right) + 0.3$ & $s_{TS} + (s_{M} -1)\lfloor\frac{K-1}{3}\rfloor$
\\ \bottomrule
\end{tabular}
\end{table}

\subsubsection{Risk-averse water resource management problem}
\label{app:risk-averse}

Table \ref{tab:water_parameters} presents the probability generation parameters used for the risk-averse water resource management problem. The most probable prior water quality state is high, whenever the agricultural activity is low and vice versa. For the forecast, $s_M = 1$ represents a small number of samples from the lake, which makes the sharpness factor small. $s_M = 3$ on the other hand generates a more accurate estimate and hence the sharpness factor is large. The cost of dilution takes the center value from $s_D$ with a moderate sharpness factor. The center of the posterior water quality is influenced heavily by the dilution and agricultural activity, and less so by the prior water quality. The used sharpness factor is moderate for the posterior water quality.

\begin{table}[h!]
\caption{Used parameters in the risk-averse water resource management problem.}\label{tab:water_parameters}
\begin{tabular}{@{}l@{\quad}c@{}@{\quad}c@{}@{\quad}c@{}}
\toprule
Node              & Information set  &   $r$ & $\bar s$       \\ \midrule
$W_1$    & $A$   &       1  & $K + 1 -s_{A}$  \\ 
\hline \\
$F$    & $W_1, M$  &       $0.5(s_m - 1) + 0.3$    &  $s_{W_1}$ \\
\hline \\
$C$    & $D, M$ &       $0.8$    &  $s_{D}$   \\
\hline \\
$W_2$    & $A, W_1, D$ & 0.8     &   $0.5 + \frac{s_{w_1} + (K-1)s_D - (K-1)s_A}{K - 1}$     \\
 \bottomrule
\end{tabular}
\end{table}

\section*{Declaration of generative AI and AI-assisted technologies in the manuscript preparation process}

During the preparation of this work the authors used Claude Sonnet 5 in order to assist in designing the probability generation methods used in the computational experiments and to process the output files of the computational experiments to collect and group results. After using this tool, the authors reviewed and edited the content as needed and take full responsibility for the content of the published article.
\printcredits

\bibliographystyle{cas-model2-names}

\bibliography{bibliography_rev1}

@article{li2024dynamic,
  title={Dynamic risk-based methodology for economic life assessment of aging subsea pipelines},
  author={Li, Xinhong and Liu, Yazhou and Chen, Guoming and Abbassi, Rouzbeh},
  journal={Ocean Engineering},
  volume={294},
  pages={116687},
  year={2024},
  publisher={Elsevier}
}

@article{howard2005influence,
  title={Influence diagrams},
  author={Howard, Ronald A and Matheson, James E},
  journal={Decision Analysis},
  volume={2},
  number={3},
  pages={127--143},
  year={2005},
  publisher={INFORMS}
}

@book{savage1954foundations,
    author ={Savage, Leonard J},
    title = {The Foundations of Statistics},
    publisher = {New York: John Wiley \& Sons},
    year = 1954
}

@article{madsen2005hugin,
  title={The {H}ugin tool for probabilistic graphical models},
  author={Madsen, Anders L and Jensen, Frank and Kjaerulff, Uffe B and Lang, Michael},
  journal={International Journal on Artificial Intelligence Tools},
  volume={14},
  number={03},
  pages={507--543},
  year={2005},
  publisher={World Scientific}
}

@article{bayesfusion2017genie,
  title={Genie modeler},
  author={BayesFusion, LLC},
  journal={User Manual. Available online: https://support. bayesfusion. com/docs/(accessed on 21 October 2019)},
  volume={16},
  pages={30--32},
  year={2017}
}

@article{gonzalez2019adversarial,
  title={Adversarial risk analysis for bi-agent influence diagrams: An algorithmic approach},
  author={Gonz{\'a}lez-Ortega, Jorge and Insua, David R{\'\i}os and Cano, Javier},
  journal={European Journal of Operational Research},
  volume={273},
  number={3},
  pages={1085--1096},
  year={2019},
  publisher={Elsevier}
}

@article{borgonovo2014decision,
  title={Decision-network polynomials and the sensitivity of decision-support models},
  author={Borgonovo, Emanuele and Tonoli, Fabio},
  journal={European Journal of Operational Research},
  volume={239},
  number={2},
  pages={490--503},
  year={2014},
  publisher={Elsevier}
}

@article{cobb2008decision,
  title={Decision making with hybrid influence diagrams using mixtures of truncated exponentials},
  author={Cobb, Barry R and Shenoy, Prakash P},
  journal={European Journal of Operational Research},
  volume={186},
  number={1},
  pages={261--275},
  year={2008},
  publisher={Elsevier}
}

@article{cobb2007influence,
  title={Influence diagrams with continuous decision variables and non-{G}aussian uncertainties},
  author={Cobb, Barry R},
  journal={Decision Analysis},
  volume={4},
  number={3},
  pages={136--155},
  year={2007},
  publisher={INFORMS}
}

@article{charnes2004multistage,
  title={Multistage {M}onte {C}arlo method for solving influence diagrams using local computation},
  author={Charnes, John M and Shenoy, Prakash P},
  journal={Management Science},
  volume={50},
  number={3},
  pages={405--418},
  year={2004},
  publisher={INFORMS}
}

@article{keating2023using,
  title={Using decision analysis to determine the feasibility of a conservation translocation},
  author={Keating, Laura M and Randall, Lea and Stanton, Rebecca and McCormack, Casey and Lucid, Michael and Seaborn, Travis and Converse, Sarah J and Canessa, Stefano and Moehrenschlager, Axel},
  journal={Decision Analysis},
  volume={20},
  number={4},
  pages={295--310},
  year={2023},
  publisher={Informs}
}

@article{shachter1986evaluating,
  title={Evaluating influence diagrams},
  author={Shachter, Ross D},
  journal={Operations Research},
  volume={34},
  number={6},
  pages={871--882},
  year={1986},
  publisher={INFORMS}
}

@article{mancuso2021optimal,
  title={Optimal Prognostics and Health Management-driven inspection and maintenance strategies for industrial systems},
  author={Mancuso, Alessandro and Compare, Michele and Salo, Ahti and Zio, Enrico},
  journal={Reliability Engineering \& System Safety},
  volume={210},
  pages={107536},
  year={2021},
  publisher={Elsevier}
}

@article{tatman1990dynamic,
  title={Dynamic programming and influence diagrams},
  author={Tatman, Joseph A and Shachter, Ross D},
  journal={IEEE Transactions on Systems, Man, and Cybernetics},
  volume={20},
  number={2},
  pages={365--379},
  year={1990},
  publisher={IEEE}
}

@article{lauritzen2001representing,
  title={Representing and solving decision problems with limited information},
  author={Lauritzen, Steffen L and Nilsson, Dennis},
  journal={Management Science},
  volume={47},
  number={9},
  pages={1235--1251},
  year={2001},
  publisher={INFORMS}
}

@article{maua2012solving,
  title={Solving limited memory influence diagrams},
  author={Mau{\'a}, Denis Deratani and de Campos, Cassio P and Zaffalon, Marco},
  journal={Journal of Artificial Intelligence Research},
  volume={44},
  pages={97--140},
  year={2012}
}

@article{maua2016fast,
  title={Fast local search methods for solving limited memory influence diagrams},
  author={Mau{\'a}, Denis Deratani and Cozman, Fabio Gagliardi},
  journal={International Journal of Approximate Reasoning},
  volume={68},
  pages={230--245},
  year={2016},
  publisher={Elsevier}
}

@article{salo2022decision,
  title={Decision programming for mixed-integer multi-stage optimization under uncertainty},
  author={Salo, Ahti and Andelmin, Juho and Oliveira, Fabricio},
  journal={European Journal of Operational Research},
  volume={299},
  number={2},
  pages={550--565},
  year={2022},
  publisher={Elsevier}
}

@article{parmentier2020integer,
  title={Integer programming on the junction tree polytope for influence diagrams},
  author={Parmentier, Axel and Cohen, Victor and Lecl{\`e}re, Vincent and Obozinski, Guillaume and Salmon, Joseph},
  journal={INFORMS Journal on Optimization},
  volume={2},
  number={3},
  pages={209--228},
  year={2020},
  publisher={INFORMS}
}

@inproceedings{cooper1988method,
  author    = {Cooper, Gregory F.},
  title     = {A Method for Using Belief Networks as Influence Diagrams},
  booktitle = {Proceedings of the Fourth Conference on Uncertainty in Artificial Intelligence (UAI-88)},
  pages     = {55--63},
  address   = {Minneapolis, MN},
  year      = {1988}
}

@misc{gurobi,
    author = {{Gurobi Optimization, LLC}},
    title = {{Gurobi optimizer reference manual}},
    year = 2022,
    url = "https://www.gurobi.com"
}

@article{herrala2025risk,
  title={Risk-averse decision strategies for influence diagrams using rooted junction trees},
  author={Herrala, Olli and Terho, Topias and Oliveira, Fabricio},
  journal={Operations Research Letters},
  pages={107308},
  year={2025},
  volume={61},
  publisher={Elsevier}
}

@article{de2025choosing,
  title={Choosing portfolios of reinforcement actions for distribution grids based on partial information},
  author={de la Barra, Joaqu{\'\i}n and Salo, Ahti and Pourakbari-Kasmaei, Mahdi},
  journal={Reliability Engineering \& System Safety},
  pages={111729},
  year={2025},
  publisher={Elsevier}
}

@article{beuzen2018comparison,
  title={A comparison of methods for discretizing continuous variables in {B}ayesian Networks},
  author={Beuzen, Tomas and Marshall, Lucy and Splinter, Kristen D},
  journal={Environmental Modelling \& Software},
  volume={108},
  pages={61--66},
  year={2018},
  publisher={Elsevier}
}

@article{weflen2022influence,
  title={An influence diagram approach to automating lead time estimation in Agile {K}anban project management},
  author={Weflen, Eric and MacKenzie, Cameron A and Rivero, Iris V},
  journal={Expert Systems with Applications},
  volume={187},
  pages={115866},
  year={2022},
  publisher={Elsevier}
}

@article{khakzad2021optimal,
  title={Optimal firefighting to prevent domino effects: Methodologies based on dynamic influence diagram and mathematical programming},
  author={Khakzad, Nima},
  journal={Reliability Engineering \& System Safety},
  volume={212},
  pages={107577},
  year={2021},
  publisher={Elsevier}
}

@article{basnet2023selecting,
  title={Selecting cost-effective risk control option for advanced maritime operations; Integration of STPA-BN-Influence diagram},
  author={Basnet, Sunil and BahooToroody, Ahmad and Montewka, Jakub and Chaal, Meriam and Banda, Osiris A Valdez},
  journal={Ocean Engineering},
  volume={280},
  pages={114631},
  year={2023},
  publisher={Elsevier}
}

@article{oni2024evaluation,
  title={Evaluation of maintenance decisions to optimize navigable inland waterway lock conditions},
  author={Oni, Bukola and Madson, Katherine and MacKenzie, Cameron},
  journal={Journal of Infrastructure Systems},
  volume={30},
  number={3},
  pages={04024013},
  year={2024},
  publisher={American Society of Civil Engineers}
}

@article{cobb2021statistical,
  title={Statistical process control for the number of defectives with limited memory},
  author={Cobb, Barry R},
  journal={Decision Analysis},
  volume={18},
  number={3},
  pages={203--217},
  year={2021},
  publisher={INFORMS}
}

@article{falconer2024eliciting,
  title={Eliciting informative priors by modeling expert decision making},
  author={Falconer, Julia R and Frank, Eibe and Polaschek, Devon LL and Joshi, Chaitanya},
  journal={Decision Analysis},
  volume={21},
  number={2},
  pages={77--90},
  year={2024},
  publisher={INFORMS}
}

@article{french2024whose,
  title={Whose judgement? {R}eflections on elicitation in {B}ayesian analysis},
  author={French, Simon},
  journal={Decision Analysis},
  volume={21},
  number={3},
  pages={143--159},
  year={2024},
  publisher={INFORMS}
}

@article{bielza2010modeling,
  title={Modeling challenges with influence diagrams: Constructing probability and utility models},
  author={Bielza, Concha and Gomez, Manuel and Shenoy, Prakash P},
  journal={Decision Support Systems},
  volume={49},
  number={4},
  pages={354--364},
  year={2010},
  publisher={Elsevier}
}

@article{varis1990bayesian,
  title={Bayesian influence diagram approach to complex environmental management including observational design},
  author={Varis, Olli and Kettunen, Juhani and Sirvi{\"o}, Hannu},
  journal={Computational Statistics \& Data Analysis},
  volume={9},
  number={1},
  pages={77--91},
  year={1990},
  publisher={Elsevier}
}

@article{rockafellar2000optimization,
  title={Optimization of conditional value-at-risk},
  author={Rockafellar, R Tyrrell and Uryasev, Stanislav and others},
  journal={Journal of Risk},
  volume={2},
  pages={21--42},
  year={2000},
  publisher={Citeseer}
}

@article{charnes1959chance,
  title={Chance-constrained programming},
  author={Charnes, Abraham and Cooper, William W},
  journal={Management Science},
  volume={6},
  number={1},
  pages={73--79},
  year={1959},
  publisher={INFORMS}
}

@book{Raiffa1968,
  title={Decision Analysis},
  author={Raiffa, Howard},
  year={1968},
  publisher={Addison-Wesley}
}

@article{smith1989influence,
  title={Influence diagrams for {B}ayesian decision analysis},
  author={Smith, James Q},
  journal={European Journal of Operational Research},
  volume={40},
  number={3},
  pages={363--376},
  year={1989},
  publisher={Elsevier}
}

@article{smith1993structuring,
  title={Structuring conditional relationships in influence diagrams},
  author={Smith, James E and Holtzman, Samuel and Matheson, James E},
  journal={Operations Research},
  volume={41},
  number={2},
  pages={280--297},
  year={1993},
  publisher={INFORMS}
}

@article{demirer2006sequential,
  title={Sequential valuation networks for asymmetric decision problems},
  author={Demirer, R{\i}za and Shenoy, Prakash P},
  journal={European Journal of Operational Research},
  volume={169},
  number={1},
  pages={286--309},
  year={2006},
  publisher={Elsevier}
}

@article{shenoy2000valuation,
  title={Valuation network representation and solution of asymmetric decision problems},
  author={Shenoy, Prakash P},
  journal={European Journal of Operational Research},
  volume={121},
  number={3},
  pages={579--608},
  year={2000},
  publisher={Elsevier}
}

@article{detwarasiti2005influence,
  title={Influence diagrams for team decision analysis},
  author={Detwarasiti, Apiruk and Shachter, Ross D},
  journal={Decision Analysis},
  volume={2},
  number={4},
  pages={207--228},
  year={2005},
  publisher={INFORMS}
}

@article{koller2003multi,
  title={Multi-agent influence diagrams for representing and solving games},
  author={Koller, Daphne and Milch, Brian},
  journal={Games and Economic Behavior},
  volume={45},
  number={1},
  pages={181--221},
  year={2003},
  publisher={Elsevier}
}

@article{rodriguez2005solving,
  title={Solving influence diagrams with fuzzy chance and value nodes},
  author={Rodr{\'i}guez-Mu{\~n}iz, Luis J and L{\'o}pez-D{\'i}az, Miguel and Gil, Mar{\'i}a {\'A}ngeles},
  journal={European Journal of Operational Research},
  volume={167},
  number={2},
  pages={444--460},
  year={2005},
  publisher={Elsevier}
}

@inproceedings{poland1993mixtures,
  title={Mixtures of {G}aussians and minimum relative entropy techniques for modeling continuous uncertainties},
  author={Poland, William B and Shachter, Ross D},
  booktitle={Uncertainty in Artificial Intelligence},
  pages={183--190},
  year={1993},
  organization={Elsevier}
}

@article{artzner1999coherent,
  title={Coherent measures of risk},
  author={Artzner, Philippe and Delbaen, Freddy and Eber, Jean-Marc and Heath, David},
  journal={Mathematical Finance},
  volume={9},
  number={3},
  pages={203--228},
  year={1999},
  publisher={Wiley Online Library}
}

@article{ahmed2007coherent,
  title={Coherent risk measures in inventory problems},
  author={Ahmed, Shabbir and {\c{C}}akmak, Ula{\c{s}} and Shapiro, Alexander},
  journal={European Journal of Operational Research},
  volume={182},
  number={1},
  pages={226--238},
  year={2007},
  publisher={Elsevier}
}

@article{roponen2020adversarial,
  title={Adversarial risk analysis under partial information},
  author={Roponen, Juho and Insua, David R{\'\i}os and Salo, Ahti},
  journal={European Journal of Operational Research},
  volume={287},
  number={1},
  pages={306--316},
  year={2020},
  publisher={Elsevier}
}

@article{capinski2015hedging,
  title={Hedging conditional value at risk with options},
  author={Capi{\'n}ski, Maciej J},
  journal={European Journal of Operational Research},
  volume={242},
  number={2},
  pages={688--691},
  year={2015},
  publisher={Elsevier}
}

@article{liu2024newsvendor,
  title={Newsvendor conditional value-at-risk minimisation: A feature-based approach under adaptive data selection},
  author={Liu, Congzheng and Zhu, Wenqi},
  journal={European Journal of Operational Research},
  volume={313},
  number={2},
  pages={548--564},
  year={2024},
  publisher={Elsevier}
}

@article{nasini2022multi,
  title={Multi-market portfolio optimization with conditional value at risk},
  author={Nasini, Stefano and Labb{\'e}, Martine and Brotcorne, Luce},
  journal={European Journal of Operational Research},
  volume={300},
  number={1},
  pages={350--365},
  year={2022},
  publisher={Elsevier}
}

@article{borgonovo2018risk,
  title={Risk analysis and decision theory: A bridge},
  author={Borgonovo, Emanuele and Cappelli, Veronica and Maccheroni, Fabio and Marinacci, Massimo},
  journal={European Journal of Operational Research},
  volume={264},
  number={1},
  pages={280--293},
  year={2018},
  publisher={Elsevier}
}

@article{shachter1989gaussian,
  title={Gaussian influence diagrams},
  author={Shachter, Ross D and Kenley, C Robert},
  journal={Management Science},
  volume={35},
  number={5},
  pages={527--550},
  year={1989},
  publisher={INFORMS}
}

@article{lopez2007influence,
  title={Influence diagrams with super value nodes involving imprecise information},
  author={L{\'o}pez-D{\'\i}az, Miguel and Rodr{\'\i}guez-Mu{\~n}iz, Luis J},
  journal={European Journal of Operational Research},
  volume={179},
  number={1},
  pages={203--219},
  year={2007},
  publisher={Elsevier}
}

@article{mcgovern1994using,
  title={Using case based reasoning for basis development in intelligent decision systems},
  author={McGovern, Jim and Samson, Danny and Wirth, Andrew},
  journal={European Journal of Operational Research},
  volume={77},
  number={1},
  pages={40--59},
  year={1994},
  publisher={Elsevier}
}

@article{regan1995r,
  title={R\&D Decision Advisor: An interactive approach to normative decision system model construction},
  author={Regan, Peter J and Holtzman, Samuel},
  journal={European Journal of Operational Research},
  volume={84},
  number={1},
  pages={116--133},
  year={1995},
  publisher={Elsevier}
}

@article{lu2009interactive,
  title={Interactive construction of graphical decision models based on causal mechanisms},
  author={Lu, Tsai-Ching and Druzdzel, Marek J},
  journal={European Journal of Operational Research},
  volume={199},
  number={3},
  pages={873--882},
  year={2009},
  publisher={Elsevier}
}

@article{borgonovo2025fifty,
  title={Fifty years of decision analysis in operational research: A review},
  author={Borgonovo, Emanuele and Jose, Victor Richmond R and Knowlton, Morgan and Shachter, Ross and Siebert, Johannes Ulrich and Ulu, Canan},
  journal={European Journal of Operational Research},
  year={2025},
  publisher={Elsevier}
}

@inproceedings{hankimaa2025solving,
  title={Solving influence diagrams: {E}fficient mixed-integer programming formulation and heuristic},
  author={Hankimaa, Helmi and Herrala, Olli and Oliveira, Fabricio and de Balsch, Jaan Tollander},
  booktitle={International Conference on Learning and Intelligent Optimization},
  pages={199--214},
  year={2025},
  organization={Springer}
}

@inproceedings{shacter1998bayes,
  title={Bayes ball: {T}he rational pastime},
  author={Shachter, R},
  booktitle={Proceedings of the 14th Annual Conference on Uncertainty in Artificial Intelligence},
  pages={24--26},
  year={1998}
}

@article{virtanen2004modeling,
  title={Modeling pilot's sequential maneuvering decisions by a multistage influence diagram},
  author={Virtanen, Kai and Raivio, Tuomas and Hamalainen, Raimo P},
  journal={Journal of Guidance, Control, and Dynamics},
  volume={27},
  number={4},
  pages={665--677},
  year={2004}
}

@article{neuvonen2025optimizing,
  title={Optimizing the {F}innish colorectal cancer population screening program with decision programming},
  author={Neuvonen, Lauri and Dillon, Mary and Vilkkumaa, Eeva and Salo, Ahti and J{\"a}ntti, Maija and Hein{\"a}vaara, Sirpa},
  journal={European Journal of Operational Research},
  year={2025},
  volume={327},
  number={1},
  pages={295--308},
  publisher={Elsevier}
}

@article{Lubin2023,
    author = {Miles Lubin and Oscar Dowson and Joaquim {Dias Garcia} and Joey Huchette and Beno{\^i}t Legat and Juan Pablo Vielma},
    title = {{JuMP} 1.0: {R}ecent improvements to a modeling language for mathematical optimization},
    journal = {Mathematical Programming Computation},
    volume = {15},
    pages = {581–589},
    year = {2023},
    doi = {10.1007/s12532-023-00239-3}
}

@article{hansen2021integrated,
  title={An integrated approach to solving influence diagrams and finite-horizon partially observable decision processes},
  author={Hansen, Eric A},
  journal={Artificial Intelligence},
  volume={294},
  pages={103431},
  year={2021},
  publisher={Elsevier}
}

@article{howard2005influenceretro,
  title={Influence diagram retrospective},
  author={Howard, Ronald A and Matheson, James E},
  journal={Decision Analysis},
  volume={2},
  number={3},
  pages={144--147},
  year={2005},
  publisher={INFORMS}
}

@article{contasti2023balancing,
  title={Balancing tradeoffs in climate-smart agriculture: Will selling carbon credits offset potential losses in the net yield income of small-scale soybean ({G}lycine max {L}.) producers in the mid-southern {U}nited {S}tates?},
  author={Contasti, Adrienne L and Firth, Alexandra G and Baker, Beth H and Brooks, John P and Locke, Martin A and Morin, Dana J},
  journal={Decision Analysis},
  volume={20},
  number={4},
  pages={252--275},
  year={2023},
  publisher={INFORMS}
}

@article{khosravi2020bayesian,
  title={Bayesian decision network-based security risk management framework},
  author={Khosravi-Farmad, Masoud and Ghaemi-Bafghi, Abbas},
  journal={Journal of Network and Systems Management},
  volume={28},
  pages={1794--1819},
  year={2020},
  publisher={Springer}
}

@article{zhang1994computational,
  title={A computational theory of decision networks},
  author={Zhang, Nevin Lianwen and Qi, Runping and Poole, David},
  journal={International Journal of Approximate Reasoning},
  volume={11},
  number={2},
  pages={83--158},
  year={1994},
  publisher={Elsevier}
}

@article{penman2020bayesian,
  title={Bayesian decision network modeling for environmental risk management: A wildfire case study},
  author={Penman, Trent D and Cirulis, Brett and Marcot, Bruce G},
  journal={Journal of Environmental Management},
  volume={270},
  pages={110735},
  year={2020},
  publisher={Elsevier}
}

@article{wan2022stay,
  title={Stay home or not? {M}odeling individuals’ decisions during the {C}{O}{V}{I}{D}-19 pandemic},
  author={Wan, Qifeng and Xu, Xuanhua and Hunt, Kyle and Zhuang, Jun},
  journal={Decision Analysis},
  volume={19},
  number={4},
  pages={319--336},
  year={2022},
  publisher={INFORMS}
}

@article{gu2024bayesian,
  title={A {B}ayesian decision network--based pre-disaster mitigation model for earthquake-induced cascading events to balance costs and benefits on a limited budget},
  author={Gu, Wenjing and Qiu, Jiangnan and Hu, Jilei and Tang, Xiaowei},
  journal={Computers \& Industrial Engineering},
  volume={191},
  pages={110161},
  year={2024},
  publisher={Elsevier}
}

@article{lu2022bayesian,
  title={Bayesian decision network--based optimal selection of hardening strategies for power distribution systems},
  author={Lu, Qin and Zhang, Wei and Hughes, William and Bagtzoglou, Amvrossios C},
  journal={ASCE-ASME Journal of Risk and Uncertainty in Engineering Systems, Part A: Civil Engineering},
  volume={8},
  number={3},
  pages={04022038},
  year={2022},
  publisher={American Society of Civil Engineers}
}



\end{document}